\documentclass[12pt]{amsart}

\usepackage{amsxtra,amssymb,amsthm,amsmath,amscd,url,mathrsfs}
\usepackage[latin1]{inputenc}
\usepackage{fullpage}
\usepackage{supertabular}

\usepackage[initials]{amsrefs}

\usepackage[
        colorlinks,
]{hyperref}
\usepackage{url}

\usepackage{color}

\renewcommand{\leq}{\leqslant}
\renewcommand{\geq}{\geqslant}

\newcommand{\ee}{{\rm e}}
\newcommand{\ii}{{\rm i}}

\newcommand{\G}{\mathbf{G}}
\newcommand{\Cc}{\mathbb{C}}

\newcommand{\Zz}{\mathbb{Z}}
\newcommand{\Pp}{\mathbb{P}}
\newcommand{\Rr}{\mathbb{R}}

\newcommand{\Qq}{\mathbb{Q}}
\newcommand{\Fp}{\mathbb{F}}

\newcommand{\ra}{\rightarrow}

\DeclareMathOperator{\nsp}{N_{Spin}}
\DeclareMathOperator{\disc}{disc}
\DeclareMathOperator{\Sym}{Sym}

\DeclareMathOperator{\Frob}{\mathrm{Frob}}
\DeclareMathOperator{\Gal}{Gal}

\DeclareMathOperator{\Aut}{Aut}
\DeclareMathOperator{\GL}{GL}

\DeclareMathOperator{\Sp}{Sp}

\DeclareMathOperator{\Oo}{O}

\newcommand{\eps}{\varepsilon}

\theoremstyle{plain}
\newtheorem{theorem}{Theorem}[section]

\newtheorem{lemma}[theorem]{Lemma}
\newtheorem{corollary}[theorem]{Corollary}

\newtheorem{proposition}[theorem]{Proposition}

\theoremstyle{remark}
\newtheorem{remark}[theorem]{Remark}

\theoremstyle{definition}

\begin{document}

\title[Independence of the zeros of $L$-functions]{Independence of the zeros of elliptic curve\\ 
$L$-functions  over function fields}

\author{Byungchul Cha}
\address{Department of Mathematics and Computer Science, Muhlenberg College, 2400 Chew st., Allentown, PA 18104, USA}
\email{cha@muhlenberg.edu}
\author{Daniel Fiorilli}
\address{Department of Mathematics and Statistics, University of Ottawa, 
585 King Edward Ave, Ottawa, Ontario, K1N 6N5, Canada}
\email{daniel.fiorilli@uottawa.ca}
\author{Florent Jouve}
\address{D\'epartement de Math\'ematiques\\ B\^atiment 425\\ Facult\'e des Sciences d'Orsay\\ Universit\'e Paris-Sud 11\\
F-91405 Orsay Cedex, France}
\email{florent.jouve@math.u-psud.fr}





\begin{abstract}
The Linear Independence hypothesis (LI), which states roughly that the imaginary parts of the critical zeros of Dirichlet $L$-functions are linearly independent over 
the rationals, is known to have interesting consequences in the study of prime number races, as was pointed out by Rubinstein and Sarnak. In this paper, we prove that a function field analogue of LI holds generically within certain families of elliptic curve $L$-functions and their symmetric powers. More precisely, for certain algebro-geometric families of elliptic curves defined over the function field of a fixed curve over a finite field, we give strong quantitative bounds for the number of elements in the family for which the relevant $L$-functions have their zeros as linearly independent over the rationals as possible.
\end{abstract}

\maketitle

\section{Motivation and the function field setting}\label{section:bias}

\subsection{Chebychev's bias: the classical case and the case of elliptic curves over $\Qq$}\label{sub:motivation}
 The prime number theorem in arithmetic progressions asserts that prime numbers are 
 asymptotically equally distributed among
 the invertible classes modulo a given integer $q\geq 1$. However Chebychev first noticed (in the case 
 $q=4$, see~\cite{Che53}) that if one only goes up to a given $x\geq 2$ the number 
 of primes congruent to $3$ modulo $4$ ``often exceeds''  the number 
 of those congruent to $1$ modulo $4$. This phenomenon called \emph{Chebychev's bias} 
 has since been extensively studied and generalized. A contemporary reference containing 
 background and presenting a systematic approach of this question is~\cite{RS94}. 
 In \emph{loc.~cit.~}Rubinstein and Sarnak explain precisely the role played by 
 the Dirichlet $L$-function $L(s,\chi)$ for primitive characters modulo $q$. Notably a
  (wide open) conjecture 
 referred to as LI (for Linear Independence, also called GSH, for Grand Simplicity Hypothesis, in~\cite{RS94})
  asserts that the multiset $\{\gamma\geq 0\colon L(1/2+i\gamma,\chi)=0\}$ 
 where $\chi$ runs over the set of primitive Dirichlet characters modulo $q$, is linearly 
 independent over $\Qq$. This assumption is shown in~\cite{RS94} to be crucial in the study of Chebychev's bias.
 
 A natural analogue from arithmetic geometry one might think of is the following. Let 
 $E/\Qq$ be an elliptic curve. One has the Sato--Tate conjecture (now a theorem thanks to~\cite{CHT08},~\cite{HSBT10}, 
 and~\cite{Tay08}) 
 that can be seen as analogous to the prime number theorem in arithmetic progressions 
 since it asserts that for any real numbers $a,b$ satisfying $0\leq a\leq b\leq \pi$, one has
 $$
 \lim_{x\rightarrow\infty}\frac{\#\{p\leq x\colon E
 \text{ has good reduction at $p$ and }\theta_p\in [a,b]\}}{\pi(x)}=\frac{2}{\pi}\int_a^b\sin^2u\,{\rm d}u\,,
 $$
 as long as $E$ is not a CM elliptic curve and
 where for a prime $p$ of good reduction we use the Hasse bound to write 
 $a_p(E):=p+1-\#E(\Fp_p)=2\sqrt{p}\cos{\theta_p}$, for a unique $\theta_p\in[0,\pi]$. 
 
 Mazur \cite{Maz08} raises the question of the existence of a bias 
 between the primes up to $x$ for which $a_p(E)$ is positive
  and those for which it is negative. Sarnak's framework to study 
  this question \cite{Sar07} turns out to be very effective,
   and explains very well this race, in terms of the zeros 
   (and potential poles) of the symmetric powers $L(\text{Sym}^n E,s)$,
    conditional on a Riemann Hypothesis and a Linear Independence assumption.
     Sarnak also remarked that a related question can be studied by considering
      the sign of the summatory function of $a_p(E)/\sqrt p$ using the zeros of $L(E,s)$ alone. 
 This function is 
 $$
 S(x)=-\frac{\log x}{\sqrt x}\sum_{p\leq x}\frac{a_p(E)}{\sqrt{p}}\,.
 $$
 The associated lower and upper densities Sarnak introduces are analogous to 
 the ones used in~\cite{RS94} in the classical setting:
 $$
 \underline{\overline{\delta}}(E)=\underline{\overline{\lim}}_{T\rightarrow\infty}\frac{1}{\log T}\int_{2}^T 
 {\bf 1}_{\{x\colon S(x)\geq 0\}}(t)\frac{dt}{t}\,.
 $$
Sarnak shows that conditionally on the Riemann Hypothesis for $L(E,s)$ and a hypothesis 
about the independence of the zeros of $L(E,s)$ the two limits coincide 
and the common value $\delta(E)$ is different from $1/2$. He also discovers a link 
between the value of $\delta(E)$ and the analytic rank of $E$. In~\cite{Fio13} the second 
author pushes this analysis 
further and shows that (assuming the above hypotheses) large analytic rank (compared to $\sqrt{\log N_E}$, where 
$N_E$ is the conductor of $E$) is
actually equivalent to high bias (i.e.~$\delta(E)$ can get arbitrarily close to $1$).

 From the above references it is clear that in both the residue classes mod $q$ and the elliptic curve settings the study of Chebychev's bias and its analogues relies on highly 
 conjectural properties of $L$-functions. Notably LI or the hypothesis of Bounded Multiplicity
 used in~\cite{Fio13} seem highly speculative given the current 
 state of our knowledge of the $L$-functions involved. 
 
 The purpose of the present paper is to give 
 a framework where unconditional analogues of LI can be proved. This setting is geometric in 
 nature (we focus on elliptic curves over 
 function fields of curves over a finite field)
  thus much more is known about the corresponding $L$-functions. Consequences 
  for analogues of the Chebychev bias in this setting are among 
  the main subjects developed in our paper~\cite{CFJ14}. 
  The relevance of studying the Chebychev bias in the function field setting was 
  first pointed out by the first named author e.g.~in~\cite{Cha08}.

\subsection{$L$-functions of elliptic curves over function fields}\label{sub:L-func}

Let $q$ be a power of a prime number $p\neq 2,3$. Let $\Fp_q$ be a field 
of $q$ elements and let $C/\Fp_q$ be a smooth projective geometrically connected curve 
of genus $g$. We define $K:=\Fp_q(C)$ to be the function field of $C$.
 Finally we fix an auxiliary
 prime $\ell\neq p$.

Let us define precisely what are the $L$-functions we are interested in. 
We follow~\cite[\S$3.1.7$]{Ulm05} to define 
the $L$-function $L(\rho, K, T)$ of any continuous, 
absolutely irreducible $\ell$-adic representation 
\[
\rho\colon G_K \longrightarrow \GL(V)
\]
of the absolute Galois group $G_K$ of $K$ in some finite dimensional 
$\Qq_{\ell}$-vector space $V$. For each $v$, we choose
 a decomposition group $D_v \subset G(K)$ and we let $I_v$ and 
 $\Frob_v$ be the corresponding inertia group and the geometric
  Frobenius conjugacy class, respectively. (We will sometimes write $\Frob_{\Fp,v}$ if the field 
  of constants $\Fp\supseteq\Fp_q$ is not obvious from context.) Then, the $L$-function $L(\rho, K, T)$
   is defined by the formal product
\begin{equation}\label{def:LFunction}
L(\rho, K, T) = \prod_v \det
\left(
1 - \rho(\Frob_v) T^{\deg v}  \mid V^{\rho(I_v)}
\right)^{-1}\,,
\end{equation}
where $V^{\rho(I_v)}$ denotes the subspace of inertia invariants.

Given an elliptic curve $E/K$ we focus on the continuous $\ell$-adic representation
\[
\rho_{\ell,E/K} : G_K \longrightarrow \Aut(V_{\ell}(E)),
\]
arising from the Galois action on $V_\ell(E) := T_{\ell}(E) \otimes \Qq_\ell$, where 
$T_\ell(E)$ is the $\ell$-adic Tate module of $E/K$. Because of a well known
 independence of $\ell$ property of the family $(\rho_{\ell,E/K})$ (namely 
  $(\rho_{\ell,E/K})_\ell$ forms a compatible system of representations),
   the $L$-function $L(\rho_{\ell, E/K},K,T)$ 
  will often be denoted simply $L(E/K,T)$ in the sequel.
  
   \par\medskip
More generally for each $m\ge1$ we may form 
\[
\Sym^m(\rho_{\ell,E/K}): G_K \longrightarrow \Aut(\Sym^m(V_{\ell}(E))),
\]
by taking the $m$-th symmetric power of $\rho_{\ell,E/K}$. Again by independence of $\ell$ 
we will write $L((\Sym^m E)/K,T)$ for the $L$-function attached to the 
representation $\Sym^m(\rho_{\ell,E/K})$.

Let us recall the explicit form of the local factors of $L((\Sym^m E)/ K, T)$. 
The local factor of 
$L((\Sym^m E)/ K, T)$ at an unramified prime $v$ is given by 
\begin{equation}\label{eq:LocalSym}
\prod_{j=0}^m(1 - {\alpha_v}^{m-j}{\beta_v}^jT^{\deg(v)} )^{-1}\,,
\end{equation}
where $\alpha_v$ ,$\beta_v$ are the (geometric) Frobenius eigenvalues at $v$
 (i.e.~the numerator of the zeta function of the fiber $E_v$ over the residue field
 $\Fp_{q^{\deg v}}$ is $L(E_v/\Fp_{q^{\deg v}},T):=1-(\alpha_v+\beta_v)T+q^{\deg v}T^2$).
 
 Let us recall deep classical facts following notably from work of Deligne 
and Grothendieck. The statement can be found (written in greater generality)
 in~\cite[\S$3.1.7$ and \S$4.1$]{Ulm05}. The deepest part (iii) of the 
 statement is a consequence of Deligne's purity result~\cite[\S3.2.3]{Del80}.
 
 \begin{theorem}\label{propL}
Assuming the $j$-invariant of $E/K$ is non-constant one has 
\begin{itemize}
\item[(i)]
$L((\Sym^m E)/ K, T)\in 1+T\Zz[T]$.
\item[(ii)]
 $L((\Sym^m E)/ K, T)$ satisfies the functional equation
 \begin{equation} \label{func-eq}
 L((\Sym^m E)/ K, T)=\eps_m(E/K)\cdot (q^{(m+1)/2}T)^{\nu_m}\cdot L((\Sym^m E)/K,1/(q^{m+1}T))\,
 \end{equation}
 where $\nu_m:=\deg L((\Sym^m E)/K,T)$ and $\eps_m(E/K)=\pm 1$. 
 Further one has for $m\geq 1$
 $$
 \nu_m=\deg \mathfrak n_m+(m+1)(2g-2)\,,
 $$
 where $\mathfrak n_m$ is the Artin conductor of the representation $\Sym^m(\rho_{\ell,E/K})$.
 If $m=1$,
 $$
 \mathfrak n_1=M+2 A\,,
 $$
  where $M$ (resp. $A$)
 denotes the locus of multiplicative (resp. additive) reduction of $E/K$.
 \item[(iii)] 
If we write 
\begin{equation}\label{eq:LfunctionEm}
 L\left((\Sym^m E)/K,T\right)=\prod_{j=1}^{\nu_m}(1-\gamma_{m,j}T),
\end{equation}
for some $\gamma_{m, j}$, then each $\gamma_{m,j}$ is of absolute value 
$q^{(m+1)/2}$ under any
 complex embedding of $\overline{\Qq_{\ell}}$. Moreover one has
 $$
 \eps_m(E/K)q^{\nu_m(m+1)/2}=\prod_{j=1}^{\nu_m}(-\gamma_{m,j})\,.
 $$
 \end{itemize}
  \end{theorem}
  
   We deduce from Theorem~\ref{propL} that we can define
   angles $\theta_{m, j}\in [0,2\pi]$ by the equation
\begin{equation}\label{eq:InverseZeroForAllM}
\gamma_{m, j} = q^{(m+1)/2}\ee^{\ii \theta_{m, j}},
\end{equation}
for all $j = 1, \dots, \nu_m$ and for each $m\ge1$.

Since our goal is to understand possible linear dependence relations among the (inverse) zeros 
of $L\left((\Sym^m E)/K,T\right)$ we first point out that~\eqref{func-eq} might impose that 
 $L\left((\Sym^m E)/K,T\right)$ vanishes at $\pm q^{-(m+1)/2}$. First we define the
  \emph{unitarized} symmetric power $L$-function of $E/K$:
  $$
  L_u((\Sym^m E)/K,T):=L((\Sym^m E)/K,T/q^{(m+1)/2})\,,
  $$
  which is a monic polynomial of $1+T\Zz[1/q^{(m+1)/2}][T]$. It is either a reciprocal or a 
  skew reciprocal polynomial:
  \begin{equation} \label{norm-func-eq}
  L_u((\Sym^m E)/K,T)=\eps_m(E/K)\cdot T^{\nu_m}\cdot L_u((\Sym^m E)/K,1/T)\,.
  \end{equation}
 This constraint is the same as the one satisfied by characteristic polynomials of isometries of 
 symmetric inner product spaces (the determinant of the 
 opposite of the isometry corresponding to the sign of the 
 functional equation). This is of course no coincidence. We handle possible 
 imposed roots (see~\cite[(4.1.2.1)]{Ulm05}) by defining \emph{reduced}
  symmetric power $L$-functions of $E/K$:
\begin{equation}\label{red-char}
 L_{\rm red}((\Sym^m E)/K,T)=\begin{cases} &L_u((\Sym^m E)/K,T)/(1+\eps_m(E/K)T)
 \,,\,\text{if}\,\,\nu_m\,\text{is odd}\,,\\
                                &L_u((\Sym^m E)/K,T)/(1-T^2)\,,\,\text{if}\,\,\nu_m\,\text{is even and}\,\,\eps_m(E/K)=-1\,,\\
                                 &L_u((\Sym^m E)/K,T)\,,\,\text{otherwise}\,.
                 \end{cases}
 \end{equation}
Note that the degree $\nu_{m,{\rm red}}$ of $L_{\rm red}((\Sym^m E)/K,T)$ is necessarily even.  


We want to study properties of linear independence over $\Qq$ of the inverse roots $\gamma_{m,j}$ 
given by~\eqref{eq:InverseZeroForAllM}, where up to reordering 
we assume that the first $\nu_{m,{\rm red}}$ roots of $L\left((\Sym^m E)/K,T\right)$ are precisely 
those of $L_{\rm red}\left((\Sym^m E)/K,T\right)$
. These algebraic integers have modulus $q^{(m+1)/2}$ (i.e.~
their ``reduced'' versions have modulus $1$) thus only 
the possible relations among their arguments are of interest. The relations we focus on are those 
of the form
$$
\prod_{j=1}^{\nu_{m,{\rm red}}}\ee^{\ii r_j\theta_{m, j}}=1\,,\qquad r_j\in \Qq\,,
$$
or equivalently after clearing denominators,
$$
\prod_{j=1}^{\nu_{m,{\rm red}}}\ee^{\ii n_j\theta_{m, j}}=1\,,\qquad n_j\in \Zz\,.
$$
In other words we wonder if the family $(1,\theta_{m,1}/2\pi,\ldots,\theta_{m,\nu_{m,{\rm red}}}/2\pi)$ 
is linearly independent over $\Qq$.
Since the main motivation of this study is to obtain meaningful results from the point 
of view of Chebychev's bias for elliptic curves over function fields, we address the more general 
question of the existence of linear relations among the $\theta_{m,j}$ \emph{as $m$ varies in 
a finite set} (as Sarnak explains the deepest results from the 
point of view of Chebychev's bias would follow from considering \emph{all} 
symmetric power $L$-functions at once but unfortunately this 
is beyond the reach of our method).
 Consequently the linear relations we are truly interested in are of the form
$$
\prod_{m=1}^k\left(\prod_{j=1}^{\nu_{m,{\rm red}}}\ee^{\ii n_{m,j}\theta_{m, j}}\right)=1\,,\qquad n_{m,j}
\in \Zz\,,
$$
where $k\geq 1$ is some fixed integer. Of course the functional equation~\eqref{func-eq} 
 translates into 
a linear dependence relation of the type above among arguments of reciprocal roots 
$\gamma_{m,j}$ (precisely these relations are given by~\eqref{eq:triv}) .
 We will call those \emph{trivial relations} and we will further denote 
$$
{\rm Rel}\left((\gamma_{m,j})_{\substack{1\leq j\leq \nu_{m,{\rm red}}\\1\leq m\leq k}}\right)
=\left\{(n_{m,j})_{\substack{1\leq j\leq \nu_{m,{\rm red}}\\1\leq m\leq k}}\colon n_{m,j}\in\Zz\text{ and } 
\prod_{m=1}^k\left(\prod_{j=1}^{\nu_{m,{\rm red}}}\ee^{\ii n_{m,j}\theta_{m, j}}\right)=1\right\}
$$
for the set of multiplicative relations among inverse roots of $L_{\rm red}(\Sym^m E/K;T)$ 
with $1\leq m\leq k$. We will say that this set is trivial if it consists only of trivial relations. Ordering the inverse roots as in~\eqref{eq:triv} the trivial relations 
are concatenations of (at most) $\nu_{m,{\rm red}}$-tuples obtained by summing row 
vectors of the shape
$$
(0,\ldots,0,1,0,\ldots 0,1,0,\ldots 0)
$$
where the two nonzero coordinates are separated by $\nu_{m,{\rm red}}/2-1$ coordinates $0$.

 We will study the existence 
of such relations among fixed families of elliptic curves over $K$. In Section~\ref{section:families} 
 we present the specific families we focus on and state our main results (Theorem~\ref{th:main1}
  and Theorem~\ref{th:main2}). Section~\ref{section:Galois} 
 can be read mostly independently of the rest of the paper: it translates (following an idea of 
 Girstmair)  the question of independence of the zeros into a question in the representation 
 theory of particular Weyl groups appearing as Galois groups over $\Qq$ of our $L$-functions. 
 In Section~\ref{section:Galois} we also give the proof of a uniform version 
 (Proposition~\ref{KoProp1.1}) of a sample of one of 
 our main results. Section~\ref{section:LS} is the technical heart of the paper. It establishes 
 general large sieve statements from which we deduce the proofs (in Section~\ref{section:proof}) 
 of our main results by appealing to big monodromy statements due to Katz.

\section{Some families of elliptic curves and generic Linear Independence}\label{section:families}

 Given a fixed elliptic curve $E/K$ with non-constant $j$-invariant we describe two ways of constructing 
 families of elliptic curves over $K$ from the base curve $E/K$.
  These families are both constructed by Katz (see~\cite{Kat02} and ~\cite{Kat05}). 
 One of the main reasons we focus on these particular families 
 is Katz's deep input asserting both these families have big monodromy 
 in a sense we will make precise later.

 \subsection{A family of quadratic twists}\label{section:quadtwists}
 
We keep the notation as in the previous section. For ease of exposition we only recall 
standard facts about quadratic twists of $E/K$ in the case where $C=\Pp^1$ (i.e. 
$K$ is the rational function field $\Fp_q(t)$). We let 
$\mathcal E\rightarrow C$ be the corresponding minimal Weierstrass model (i.e.~the 
identity component of the N\'eron model of $E$). This model is obtained by gluing the affine 
part of $E/K$ given, say, 
by the Weierstrass equation $y^2=x^3+ax+b$, where $a,b\in\mathcal \Fp_q[t]$, together with a similar model 
``at infinity''. For 
each $f\in K^\times$ we consider
$$
E_f\colon y^2=x^3+f^2ax+f^3b\,
$$
which is a Weierstrass equation for an elliptic curve over $K$. The extension $K(\sqrt{f})/K$ 
is the smallest over which $E$ and $E_f$ are isomorphic (see e.g.~\cite[Lemma $2.4$]{BH12}). 
Thus $E$ and $E_f$ are isomorphic over $K$ if and only if $f$ is a square in $K$. 
A \emph{quadratic twist} 
of $E/K$ is an elliptic curve $E_f/K$ such that $f$ is not a square in $K$. 
Note that $E_f$ is isomorphic to $E_g$ over $K$ if  
and only if there exists $c\in K^\times$ such that $f=gc^2$.
 
Let $\Delta\in\Fp_q(t)$ be the discriminant of $E/K$; then
$E_f/K$ has discriminant $f^6\Delta$. Therefore, away from 
the irreducible factors of $f$, the curves $E$ and $E_f$ have the same locus of good reduction. 
Let $v$ be a place of good reduction for $E$ and $E_f$. A crucial feature of quadratic twists is that 
for any $f\in K^\times$ one has (see e.g.~\cite[\S 2.4]{BH12})
\begin{equation}\label{eq:local-twist}
L(E_{f,v}/\Fp_{q^{\deg v}},T)=L\left(E_v/\Fp_{q^{\deg v}},\left(\frac{f}{v}\right)T\right)\,,
\end{equation}
where $(\frac{\cdot}{v})$ denotes the Legendre symbol of $\Fp_{q^{\deg v}}$. 
From the representation theoretic 
point of view the $L$-function of a quadratic twist $E_f/K$ can be
 defined as the $L$-function of $\rho_{\ell,E/K}\otimes 
\chi_f$ where $\chi_f$ is the unique nontrivial $K$-automorphism of $K(\sqrt{f})$.
 This point of 
view makes~\eqref{eq:local-twist} obvious. 

Let us now assume that $\mathcal E\rightarrow C$ has at least one fiber 
of multiplicative reduction and fix a nonzero element $m\in \mathcal \Fp_q[t]$ 
which vanishes at at least one point of the locus $M$ of multiplicative reduction of 
$\mathcal E\rightarrow C$.
The ``twisting family'' we consider was first introduced by Katz. It is the $(d+1)$-dimensional 
affine variety for 
which the $\Fp$-rational points are, 
for any algebraic extension $\Fp\supseteq \Fp_q$,
\begin{equation}\label{twisting-space}
\mathcal F_d(\Fp)=\{f\in \Fp[t]\colon f\text{ squarefree},\,\deg f=d,\, {\rm gcd} (f,m)=1\}\,,
\end{equation}
where $d\geq 1$ is an integer. A crucial fact in 
view of the study we have in mind is the following: 
if $f\in\mathcal F_d(\Fp_{q^n})$ then $\deg L(E_f/K,T)$ only depends
 on $d$ and $q$ (in particular it is independent of 
$n$ so that ultimately we will let $n\rightarrow \infty$). 

Fix $d\geq 2$. One may consider $1$-parameter 
subfamilies of $\mathcal F_d$ for which the kind of generic property 
of independence of the zeros we have in mind can be quite easily drawn from 
known results. A nice feature of these $1$-parameter families is that 
they make it possible to keep track of uniformity issues with respect to the parameters.
Fix $\tilde{f}\in \mathcal F_{d-1}(\Fp_q)$. The family we consider is the open affine
 curve $U_{\tilde{f}}$ with geometric points:
$$
U_{\tilde{f}}(\overline{\Fp_q})=\{c\in\overline{\Fp_q}\colon 
(c-t)\tilde{f}(t)\in{\mathcal F_d}(\overline{\Fp_q})\}\,.
$$
If $c\in U_{\tilde{f}}(\Fp_q)$ we denote by $E_c$ (resp. $(\gamma_{1,j}(c))_{1\leq j\leq N_{\rm red}}$)
 the quadratic twist of $E$ by $f$
 (resp. the multiset of inverse roots of its reduced $L$-function, the degree of which we denote $N_{\rm red}$) where $f(t)=(c-t)\tilde{f}(t)$. 

For this subfamily of twists we can now state a sample result of ``generic'' 
linear independence of inverse roots in the case $k=1$ (i.e.~
only $L(E_c/K,T)$ is considered).

\begin{proposition}\label{KoProp1.1}
 With notation as above, there exists integers $d_0(E)$, $q_0(E)$ 
 depending only on $E$ such that for any $\deg L(E_c/K,T):=N\geq 5$, any
 $d\geq d_0(E)$, and any $q\geq q_0(E)$, the set of relations
between zeros of the reduced $L$-functions $L_{\rm red}$ satisfies:
$$
\#\left\{c\in U_{\tilde{f}}(\Fp_q)\colon 
{\rm Rel}\left(\left(\gamma_{1,j}(c)\right)_{1\leq j\leq N_{\rm red}}\right)
\text{ is nontrivial }\right\}\ll N^2q^{1-\gamma^{-1}}\log q\,,
$$
where the  implied constant depends only on the $j$ invariant of $E$ and, 
in a controlled way, on the genus $g$ of $C/\Fp_q$, and where one can choose 
$2\gamma =7N^2-7N+4$.
\end{proposition}

We will see how this result can be deduced from 
the third author's result~(\cite[Th.~4.3]{Jou09} which relies in turn on 
a result of Hall~\cite{Hal08}) together with general group theoretic 
arguments. Before presenting our two main results (one of which is a generalization 
of Proposition~\ref{KoProp1.1}) let us give a concrete incarnation of the above statement 
in the case where the base elliptic curve $E/K$ is the Legendre curve.
  Let $K=\Fp_q(t)$ be the rational function field over $\Fp_q$. We call \emph{Legendre elliptic 
  curve} the curve $E_{\mathcal L}$ given by the Weierstrass equation
  $$
 y^2=x(x-1)(x-t)\,.
  $$
  
  Let $\mathcal F_{\mathcal L,d}$ be the corresponding twisting space~\eqref{twisting-space}. 
  For any field 
  extension $\Fp/\Fp_q$ the set of $\Fp$-rational points of this affine variety is
  $$
  \mathcal F_{\mathcal L,d}(\Fp)=\left\{P\in \Fp[t]\colon P\text{ squarefree, }\deg P=d,
  \, {\rm gcd}\left(P,t(t-1)\right)=1\right\}\,.
  $$
  As recalled in~\cite[(9)]{Jou09} (see the references therein for a proof) we have for any quadratic twist 
  $E_{\mathcal L,f}$ of $E_{\mathcal L}$ by $f\in \mathcal F_{\mathcal L,d}(\Fp_q)$:
  $$
  N:=\deg L(E_{\mathcal L,f}/K,T)=\begin{cases} 2d & \text{ if $d$ is even, }\\ 
                                                             2d-1& \text{ if $d$ is odd, } \end{cases}\,
  $$
  which is an integer independent of $f$, 
  as expected.
  
  Let us fix an integer $d\geq 3$ and an $\Fp_p$-rational element 
  $\tilde{f}\in \mathcal F_{\mathcal L,d-1}(\Fp_p)$.
   An immediate consequence of Proposition~\ref{KoProp1.1}
    combined with~\cite[Th.~$4.7$]{Jou09} is the following.

   \begin{corollary}\label{cor:Legendre-twists}
   With notation as in Proposition~\ref{KoProp1.1} one has for any $d\geq 3$ and any power 
   $q$ of $p$:
  $$
   \#\{c\in\Fp_q\colon\, \tilde{f}(c)\neq 0,\, c\neq 0,1, 
   {\rm Rel}\left(\left(\gamma_{1,j}(c)\right)_{1\leq j\leq N_{\rm red}}\right) \text{ is nontrivial } \}
   \ll d^2 2^{n_{\tilde{f}}} q^{1-\gamma^{-1}}\log q\,,
   $$
   with an absolute implied constant, where $n_{\tilde{f}}$ is a non-negative integer depending only on 
   $\tilde{f}$, and 
   where we can choose $2\gamma=7N^2-7N+4$.
   \end{corollary}

Interestingly, recent work of Ulmer~\cite{Ulm14} focuses on the quadratic twist of $E_\mathcal L$
 by $-1$ (it is isomorphic to $E_\mathcal L$ in case $-1$ is a square in $K$) and shows that over
  a suitable extension $\tilde{K}/K$ the situation regarding $L$-functions is in sharp contrast 
  with what one might expect when looking at Corollary~\ref{cor:Legendre-twists}. Indeed Ulmer shows 
  in~\cite[Prop.~10.1]{Ulm14} that $L(E_{\mathcal L, -1}/\tilde{K},T)$ is a power of 
  $1-qT$ which means the phenomenon at the exact opposite of linear independence 
  occurs for the quadratic twist 
  $E_{\mathcal L, -1}/\tilde{K}$.

One of our main goals is to generalize Proposition~\ref{KoProp1.1} in two different ways. First we 
no longer restrict to a parameter variety of dimension $1$ but we consider quadratic twists 
by any $f\in \mathcal F_d(\Fp_q)$. Also we obtain a result of linear independence for the inverse 
zeros of an arbitrary (finite) number of \emph{odd} symmetric power $L$-functions of twists $E_f$ 
at once. The result is as follows.

\begin{theorem}\label{th:main1}
 Let $K=\Fp_q(C)$ be the function field of a smooth geometrically irreducible curve 
 $C/\Fp_q$. Let $E/K$ be an elliptic curve with non-constant 
$j$-invariant and whose minimal Weierstrass model $\mathcal E\rightarrow C$ has at least one 
fiber of multiplicative reduction. 
 If $f\in\mathcal{F}_d(\Fp_{q^n})$, let $\nu_{m}$ be the degree (depending only on $q$ 
  and $\deg f=d$) of $L(\Sym^m E_f/K,T)$, the $m$-th symmetric
  power $L$-function of the twist $E_f$ of $E$ over $K$. As before let 
  $(\gamma_{m,j}(f))_{1\leq j\leq \nu_m}$ be the set of inverse roots 
  of $L(\Sym^m E_f/K,T)$ (seen as a $\Qq$-polynomial of degree $\nu_m$) ordered as in~\eqref{eq:triv}. 
  Let $k\geq 1$ be a fixed integer.
  Then for all $p$ bigger than a constant depending only on $d$ and $k$, for all 
  big enough $p$-power 
  $q:=p^n$ (precisely $n$ is bigger than a constant depending only on 
  $\overline{\mathcal{F}_d}:=\mathcal F_d\times\overline{\Fp_p}$) and for all $d$ bigger than an absolute constant,
  $$
\#\left\{f\in \mathcal F_d(\Fp_q)\colon \,{\rm Rel}\left(\left(\gamma_{2m-1,j}(f)\right)
_{\substack{1\leq j\leq \nu_{2m-1,{\rm red}}\\ 1\leq m\leq k}}\right)
\text{ is nontrivial }\right\}\ll q^{d+1-\gamma^{-1}}\log q\,,
$$
where one can take $2\gamma=4+7\sum_{m=1}^k\nu_{2m-1}(\nu_{2m-1}-1)$ and where 
the implied constant depends only on $d$ and $k$.
\end{theorem}

Let us mention that our method cannot be generalized to produce a result
 where even symmetric power $L$-functions are involved 
 (see Lemma~\ref{lem:operationscommute}). Indeed looking
 at~\eqref{eq:local-twist} and~\eqref{eq:LocalSym} it becomes obvious 
 that the local factor at a place of good reduction of, say, the $2m$-th symmetric power 
 $L$-function of a quadratic twist of $E$ coincides with the local factor 
 of the $2m$-th symmetric power $L$-function of $E$ at the same place. 
 In other words we would lose the crucial fact that we consider a \emph{family} 
 of elliptic curves and we would be left with many repetitions of a single
$L$-function in which case LI is trivially false.

 \subsection{A pullback family of elliptic curves}\label{section:pullback}
 
 This family is considered by Katz in~\cite[\S $7.3$]{Kat05}. The elliptic curve we start with 
 is a curve $E/\Fp_q(t)$ (with non-constant $j$-invariant) given, say, by a Weierstrass equation of the form
 \begin{equation}\label{eq:gen-weierstrass}
 E\colon y^2+a_1(t)y+a_3(t)xy=x^3+a_2(t)x^2+a_4(t)x+a_6(t)\,,
 \end{equation}
 where the $a_i(t)$'s are elements of $\Fp_q(t)$. Now given any non-constant function 
 $f\in\Fp_q(C)$ we may form the pullback curve
 $$
 E^f\colon y^2+a_1(f)y+a_3(f)xy=x^3+a_2(f)x^2+a_4(f)x+a_6(f)\,,
 $$
obtained by substituting $t$ by $f$ in the equation defining $E$. This defines 
an elliptic curve $E^f/\Fp_q(C)$. The construction implies deep links 
between the $L$-functions of $E/\Fp_q(t)$ and $E^f/\Fp_q(C)$. More precisely 
Katz explains ~(\cite[$(7.3.9)$]{Kat05}) that for any $n\geq 1$ one has the divisibility 
relation between rational polynomials:
$$
L((\Sym^n E)/\Fp_q(t),T)\mid L((\Sym^n E^f)/\Fp_q(C),T)\,.
$$
Of course such a divisibility relation has to be taken into account when 
studying the potential $\Qq$-linear independence of the inverse zeros of $L((\Sym^n E^f)/\Fp_q(C),T)$
as $f$ varies. Katz defines the \emph{new part} of the symmetric power $L$-function of 
$E^f/\Fp_q(C)$:
$$
L^{\rm new}((\Sym^n E^f)/\Fp_q(C),T):=\frac{L((\Sym^n E^f)/\Fp_q(C),T)}{L((\Sym^n E)/\Fp_q(t),T)}\,.
$$
Relevant to our study is the existence of linear dependence relations among 
the inverse zeros of $L^{\rm new}((\Sym^n E^f)/\Fp_q(C),T)$, or more generally 
of the product over $n$ of such $L$-functions with $1\leq n\leq k$ and $k$ fixed.
To state our main result concerning the above pullback family of elliptic curves 
let us recall the following.

A \emph{divisor} $D$ on $C$ is a formal finite $\Zz$-linear combination of 
rational points 
on $C$. An \emph{effective divisor} is one non negative coefficients. The \emph{degree} 
of a divisor is the sum (in $\Zz$) of its coefficients. To each $f\in \Fp_q(C)$
 one can associate a divisor
$$
(f):=\sum_{P}{\rm ord}_P(f)\cdot P\,,
$$
where the sum is over rational points on $C$ and ${\rm ord}_P(f)$ denotes the natural 
valuation of $f$ at $P$. To any divisor $D$ on $C$ one may attach the Riemann--Roch space
$$
\mathcal L(D):=\{f\in K^\times\colon (f)+D\geq 0\}\cup\{0\}\,.
$$
The Riemann--Roch Theorem asserts that the dimension $\ell(D)$ of $\mathcal L(D)$ is finite.

\begin{theorem}\label{th:main2}
 Let $K=\Fp_q(C)$ be the function field of a smooth geometrically irreducible curve 
 $C/\Fp_q$ of genus $g$. Let $E/K$ be an elliptic curve with non-constant 
$j$-invariant and whose minimal Weierstrass model $\mathcal E\rightarrow C$ has at least one 
fiber of multiplicative reduction. 

Let $D$ be an effective divisor on $C$ of degree at least $2g+3$ and let $U_{D,S}$ 
be the dense open subset of $\mathcal L(D)$
 defined in \S\ref{section:cohgen}. Let $n\geq 1$.
 If $f\in U_{D,S}(\Fp_{q^n})$ let 
  $(\gamma_{m,j}(f)^{\rm new})_{1\leq j\leq \nu_m}$ be the set of inverse roots 
  of $L^{\rm new}((\Sym^n E^f)/\Fp_{q^n}(C),T)$ (seen as a $\Qq$-polynomial of degree 
  $\nu_m$ depending only on $D$ and $q$). Let $k\geq 1$ be a fixed integer.
  Then for all $p$ larger than a constant depending only 
  on $\deg D$ and $k$, for all big enough $p$-power $q:=p^r$ (precisely 
  $r$ has to be bigger than a constant depending only on $D$), and for all $D$ of degree 
  larger than an absolute constant, one has
  $$
\#\left\{f\in U_{D,S}(\Fp_q): {\rm Rel}\left(\left(\gamma_{m,j}(f)^{\rm new}\right)
_{\substack{1\leq j\leq \nu_{m,{\rm red}}\\ 1\leq m\leq k}}\right)
\text{ is nontrivial }\right\}\ll q^{\ell(D)-\gamma^{-1}}\log q\,,
$$
where one can take 
$$
2\gamma=4+7\sum_{j=1}^{k} h(j),\qquad h(j):=\begin{cases} \nu_j(\nu_j-1) & \text{ if } j\text{ is odd}\,, \\
\nu_j(\nu_j+1) &\text{ if } j\text{ is even}\,,
\end{cases}
$$
 and where
the implied constant depends only on $D$ and $k$.
\end{theorem}

 For both the quadratic twist family and the pullback family the strategy of proof of Theorem~\ref{th:main1} and 
 Theorem~\ref{th:main2}
 relies on a representation theoretic interpretation of linear independence relations between 
 the roots. The idea of using the Galois action on the set of relations to study them 
 goes back to Girstmair (see references in~\cite{Kow08b}). The proofs of our results follow these ideas together with a
  sieving procedure as performed by Kowalski in~\cite{Kow08b} (where similar questions of 
  independence of zeros are addressed in the context of algebro-geometric families 
  of hyperelliptic curves over finite fields).

\section{The Galois theoretic approach to independence of the zeros}\label{section:Galois}

Let us now describe the strategy we use to attack 
the general question of linear independence of zeros of $\Qq$-polynomials.

\subsection{The general setup}\label{section:GaloisSetup} 
Fix an 
integer $k\geq 1$ and 
polynomials $P_1,\ldots, P_k$ with coefficients in a field $E$ satisfying $\Qq\subset E\subset \Cc$. 
For each $i\in\{1,\ldots k\}$ 
let $K_i$ be the splitting field of $P_i/\Qq$.  We denote by $M_i$ the set of complex roots of $P_i$ 
and we view 
$G_i:=\Gal (K_i/\Qq)$ as a subgroup of permutations of $M_i$. Assume further that the number 
fields $K_i$ are jointly 
linearly disjoint so that $P:=P_1\cdots P_k$ has splitting field (over $\Qq$) 
with Galois group isomorphic to $G:=G_1\times\cdots \times G_k$. 
Finally let $M$ be the (necessarily disjoint) union of the $M_i$'s and let $F(M)$ be the 
permutation representation of $G$ associated to the action of $G$ on the roots of $P$.

 We are interested in the question of $\Zz$-multiplicative
  independence of the zeros of $P$. Denote by 
$\langle M\rangle$ the multiplicative 
abelian group (or $\Zz$-module) generated by $M$. Set 
$\langle M\rangle_{\Qq}:=\langle M\rangle\otimes_{\Zz}\Qq$ the $\Qq$-vector space 
obtained by extension of scalars.
The vector space $\langle M\rangle_{\Qq}$ is equipped with a $G$-module 
structure (inherited from the Galois action on the roots). More precisely one has a 
$G$-equivariant linear map:
$$
r\colon F(M) \rightarrow \langle M\rangle_{\Qq}\,,
$$ 
with kernel the $G$-module of multiplicative relations 
${\rm Rel}_{\Qq}(M):={\rm Rel}(M)\otimes \Qq$. Recall that we denote, as in~\cite{Kow08b}:
$$
{\rm Rel}(M)=\{(n_\alpha)\in\Zz^M: \prod_{\alpha\in M}\alpha^{n_\alpha}=1\}\,.
$$

Note that it makes more sense when defining ${\rm Rel}(M)$ to assume the elements of $M$ 
have modulus $1$ (it is indeed the case in the application to 
$L$-functions we are interested in since we consider unitarized versions of these $L$-functions).
The crucial point is that if the $G$-module structure of $F(M)$ is known, one can 
hopefully deduce the $G$-module structure of ${\rm Rel}_{\Qq}(M)$.

\subsection{The maximal Galois group of $L$-functions}\label{section:GaloisMax}

 The elliptic curve $L$-functions we are interested in satisfy a functional equation of
  type~\eqref{func-eq}. 
 Besides the (already discussed) fact that this may impose roots, some relations (we have called 
 \emph{trivial relations}) are also imposed. The functional equation~\eqref{func-eq} 
 satisfied by $L((\Sym^m E)/K,T)$ implies multiplicative relations:
\begin{equation}\label{eq:triv}
 \gamma_{m,j}\gamma_{m,j+(\nu_{m,{\rm red}}/2)}=q^{m+1}\,,\qquad 1\leq j\leq \nu_{m,{\rm red}}/2\,,
 \end{equation}
 up to reordering the roots of $L_{\rm red}((\Sym^m E)/K,T)$. Let $g:= \nu_{m,{\rm red}}/2$. 
 Because of the above relations the Galois 
 group of the splitting field of the polynomial $L_{\rm red}((\Sym^m E)/K,T)$ over $\Qq$, seen 
 as a subgroup of the symmetric group $\mathfrak S_{2g}$ on the set of $2g$ symbols 
 $$
 M:=\{-g,\ldots,-1,1,\ldots,g\}\,,
 $$
 embeds in the group $W_{2g}$ defined by either of the following 
 equivalent conditions.
\begin{enumerate}
 \item $W_{2g}$ is the set of permutations of $2g$ letters that commute
 to a given involution $c\in \mathfrak S_{2g}$ acting without fixed points,
\item the group $W_{2g}$ is the subgroup of permutations of $M$ 
acting on pairs $\{i,-i\}$. This group fits the exact sequence:
$$
\begin{CD}
1@>>> \{\pm 1\}^g@>>> W_{2g}@>>> \mathfrak S_{g}@>>> 1\,.
\end{CD}
$$
\item $W_{2g}$ is the Weyl group of the algebraic group ${\Sp}(2g)$, i.e.~
the Weyl group corresponding to the root system of type $C_g$.
\end{enumerate}

 In other words an element $\sigma\in W_{2g}$ permutes the couples $(i,-i)$, $1\leq i\leq g$ 
 and also allows permutations within each couple. 
The latter permutation is called a \emph{sign change}. A subgroup of $W_{2g}$ of 
particular interest is what can be seen 
as its \emph{positive part}: it acts on pairs $\{i,-i\}$ but only allowing evenly many sign changes. 
In other words if one defines a signature homomorphism 
${\rm sgn}: W_{2g}\ra \{\pm 1\}$ by ${\rm sgn}\,(\sigma)=(-1)^{\#\{\text{ sign changes in }\sigma\}}$, 
then one has an exact sequence
$$
\begin{CD}
1@>>> W_{2g}^+@>>> W_{2g}@>{\rm sgn}>> \{\pm 1\} @>>> 1\,.
\end{CD}
$$ 
Conceptually the group $W_{2g}^+$ is the Weyl group of the root system of type $D_g$. 
See~\cite[end of \S1]{Kat12} for useful comments and explanations on the expected 
Galois group in our context.

As explained in Section~\ref{section:GaloisSetup} knowledge of the representation theory 
of the Galois groups of the $L$-functions considered will be crucial. Let us therefore 
state a few important facts about the action of $W_{2g}^+$ on $M\times M$.

\begin{lemma}\label{ActW+}
Assume $g\geq 3$. With notation as above:
\begin{enumerate}
 \item[(i)] there are exactly three orbits in the action of $W_{2g}^+$ on $M\times M$:
$$
\Delta=\{(i,i): i\in M\}\,,\qquad \Delta_c=\{(i,-i): i\in M\}\,,\qquad O=\{(i,j): i,j\in M, i\neq\pm j\}\,;
$$ 
\item[(ii)] let $F(M)$ be the permutation representation space associated to the action of 
$W_{2g}^+$ on $M$. Let $(f_i)_i$ be the associated formal basis. The 
decompostion of $F(M)$ as a direct sum of irreducible representations of $W_{2g}^+$ is 
$$
F(M)={\bf 1}\oplus G(M)\oplus H(M)\,,
$$
where 
\begin{align*}
G(M)&=\left\{\sum_{i\in M}t_\alpha f_i\colon t_i=t_{-i},\, i\in M, 
\text{ and }\sum_{i\in M}t_i=0\right\}\,,\\
H(M)&=\left\{\sum_{i\in M}t_i f_i\colon t_i=-t_{-i},\,i\in M\right\}\,.
\end{align*}
\end{enumerate}
\end{lemma}

\begin{proof}
 Let us start with (i). The fact that $\Delta$ is a single orbit comes from the transitivity of the action of $W_{2g}^+$ on 
 $M$. Next 
pick $(i,-i)$ and $(j,-j)$ in $\Delta_c$ and assume $1\leq i,j \leq g$. 
Obviously the permutation $\sigma\in W_{2g}$ satisfying 
$\sigma(i)=j$ (and thus $\sigma(-i)=-j$) and fixing every other element of $M$ is an 
element of $W_{2g}^+$ since the number of sign changes of $\sigma$ is $0$. 
Now fix an element $k\in M\setminus\{i,j\}$, $1\leq k\leq g$. 
This is possible since $g\geq 3$. Define $\tilde{\sigma}$ to be the permutation of $W_{2g}$
 such that $\tilde{\sigma}(i)=-j$ (and thus $\tilde{\sigma}(j)=-i$), $\tilde{\sigma}(k)=-k$, and 
 that fixes every other element of $M$. Its number of sign changes is $2$ 
therefore $\tilde{\sigma} \in W_{2g}^+$. 

Now we come to $O$. First notice that if $(\alpha,\beta)\in O$, then $(-\alpha,\beta)\in O$ as well. 
To see this, pick $\gamma\in M\setminus\{\pm \alpha,\pm \beta\}$ 
(recall $g\geq 3$) 
and set $\sigma(\alpha)=-\alpha$, $\sigma(\gamma)=-\gamma$ and $\sigma$ commutes with the 
sign change and restricts to identity outside of $\{\pm \alpha,\pm \gamma\}$. 
By construction $\sigma\in W_{2g}^+$ and $\sigma(\alpha,\beta)=(-\alpha,\beta)$. 

Fix an element $y=(i,j)\in O$, with $1\leq j\leq g$, as well as an element $k\in M\setminus\{\pm i\}$. 
Then $(i,k)\in O$. Indeed 
if $1\leq k\leq g$ then the permutation $\sigma\in W_{2g}$ such 
that $\sigma(i)=i$, $\sigma(j)=k$ (therefore $\sigma(-j)=-k$) 
and that fixes every other element of $M$ is in the kernel of ${\rm sgn}$. 
Whereas if $-g\leq k\leq -1$, then define 
$\tilde{\sigma}\in W_{2g}$ to be the permutation satisfying 
$\tilde{\sigma}(j)=k$ (therefore $\tilde{\sigma}(-j)=-k$), $\tilde{\sigma}(i)=-i$, 
and fixing every other element of $M$. 
The number of sign changes of $\tilde{\sigma}$ is two so $\tilde{\sigma}\in W_{2g}^+$. 
One has $\tilde{\sigma}(i,j)=(-i,k)$. By the above 
remark we deduce in turn $(i,k)\in O$. Finally if $-g\leq j\leq -1$
 the same line of reasoning as above applies as well. 

An easy adaptation of the above argument produces for any $k\in M\setminus\{\pm j\}$ a 
permutation $\sigma\in W_{2g}^+$ 
such that $\sigma(y)=(k,j)$. We can now prove that $O$ is a single 
$W_{2g}^+$-orbit: let $(i',j')\in O$. There exists 
$\sigma_1\in W_{2g}^+$ such that $\sigma_1(y)=(i,j')$ (provided $j'\neq \pm i$; 
otherwise $(i,j)$ can first be mapped to $(j,i)$ and then 
to $(i',i)$ or $(i',-i)$ by possibly composing with one extra permutation of $W_{2g}^+$)
 and there exists $\sigma_2\in W_{2g}^+$ such that $\sigma_2\sigma_1(y)=\sigma_2(i,j')=(i',j')$.

\par\medskip
Now we turn to (ii). The three spaces ${\bf 1}$, $G(M)$ and $H(M)$ are clearly $W_{2g}^+$-spaces. 
Let $\chi$ be the character of $F(M)$ as 
a $W_{2g}^+$-representaion. It is enough to show that $\langle \chi,\chi\rangle=3$ to prove (ii).
 Since $\chi$ is real-valued 
one has $\langle\chi,\chi\rangle=\langle\chi^2,{\bf 1}\rangle$ and this last quantity is nothing but
 the number of $W_{2g}^+$-orbits of $M\times M$ which 
we saw is three.
\end{proof}

Let us now assume $g\geq 3$ and let
$\mathcal W_{2g}$ be a group satisfying
$$
W_{2g}^+\subseteq \mathcal W_{2g}\subseteq W_{2g}\,.
$$
Since $[W_{2g}:W_{2g}^+]=2$ this means that either $\mathcal W_{2g}=W_{2g}^+$ or 
$\mathcal W_{2g}=W_{2g}$. An important point is that even though $\mathcal W_{2g}$ is not 
completely determined, its natural permutation representation is. 
Indeed Lemma~\ref{ActW+} and~\cite[Lemma $2.1$]{Kow08b} show that the 
$\mathcal W_{2g}$-module $F(M)$ has the same 
decomposition as a direct sum of irreducible $\mathcal W_{2g}$-modules,
 whichever of the two groups $W_{2g}$, $W_{2g}^+$ the group $\mathcal W_{2g}$ be. 

As a consequence~\cite[Cor.~$2.3$]{Kow08b} holds if one replaces $W_{2g}$ with the 
$k$-fold cartesian product of $\mathcal W_{2g}$. Let us state 
the result in this case.

\begin{corollary}\label{KoCor2.3}
 Let $k\geq 1$ and $g_i\geq 3$  be integers ($1\leq i\leq k$). Let ${\mathcal W}^{(k)}$ 
 be the product $\mathcal W_{2g_1}\times\cdots\times\mathcal W_{2g_k}$ of $k$ 
groups of type $\mathcal W$, (this means that for each $i$, one has
 $W_{2g_i}^+\subseteq \mathcal W_{2g_i}\subseteq  W_{2g_i}$,
 where the $j$-th copy is seen as a permutation group of a set $M_j$). 
 The group ${\mathcal W}^{(k)}$ 
acts naturally on the disjoint union $M$ of the $M_j$'s (its $j$-th factor $\mathcal W_{2g_j}$
 acts trivially on $M_i$ as long as $i\neq j$). 
Let $F(M)$ be the permutation representation corresponding to the action 
of ${\mathcal W}^{(k)}$ on $M$. It is a $(2\sum_i{g_i})$-dimensional 
${\mathcal W}^{(k)}$-module whose decomposition as a direct sum of 
(geometrically) irreducible ${\mathcal W}^{(k)}$-modules is isomorphic to
$$
{\bf 1}\oplus\bigoplus_{1\leq i\leq k}G(M_i)\oplus\bigoplus_{1\leq j\leq k}H(M_j)\,.
$$ 
\end{corollary}

\begin{proof}
 This is a direct consequence of Lemma~\ref{ActW+}, of~\cite[Lemma $2.1$]{Kow08b}, 
 and of the fact that for any finite groups $G_1$, $G_2$, the direct sum of 
an irreducible $G_1$-module by an irreducible $G_2$-module is an irreducible 
$(G_1\times G_2)$-module.
\end{proof}

Let us finally give the decomposition of ${\rm Rel}_{\Qq}(M)$ as a $G$-module.

\begin{proposition}\label{KoProp2.4}
 We keep the notation as in Corollary~\ref{KoCor2.3}. Let $k\geq 1$ and $g\geq 3$ be integers.
  Let $P_1,\ldots, P_k$ be polynomials 
such that for each $i$ the Galois group of the splitting field of $P_i$
 over $\Qq$ is isomorphic to $\mathcal W_{2g_i}$.  Let $M$ be the union of the roots of the 
 polynomials $P_i$, $1\leq i\leq k$. 
Assume that if $\alpha,\bar{\alpha}$ are 
elements of $M$ such that $(\alpha, \bar{\alpha})$ is an element of the set acted on by 
$\mathcal W^{(k)}$
 then $\alpha\bar{\alpha}\in\Qq^\times$. Then one has
$$
{\rm Rel}_{\Qq}(M)=\bigoplus_{1\leq j\leq k}{\rm Rel}_{\Qq} (M_j)\,.
$$
Moreover if $\alpha\bar{\alpha}$ is independent of $\alpha$ (say, it equals some constant 
$\mu\in\Qq$), 
then for $g\geq 5$ (or $g\geq 3$ 
if $\mu=1$) we have for each $j$:
$$
{\rm Rel}_{\Qq}(M_j)=\begin{cases}
                                           \mathbf{1}\oplus G(M_j) &\text{ if } \mu=1\,,\\
                                            G(M_j) &\text{ otherwise}\,.
                                             \end{cases}
$$
\end{proposition}

\begin{proof}
The argument is the same as in~\cite[Prop.~$2.4(1)$]{Kow08b}. In particular, 
to exclude the possibility 
that $H(M_j)$ be a sub-$\mathcal W_{2g_j}$-representation of $F(M_j)$ we appeal to a 
group theoretic argument. If $g\geq 5$ the alternating group on 
five letters appears in the composition series of $W_{2g}^+$ and so 
$W_{2g}^+$ is not solvable. Moreover $W_{2g}^+$ is not abelian if $g\geq 3$.

\end{proof}

\subsection{The key implication and the proof of Proposition~\ref{KoProp1.1}}\label{section:implication}
 
 For any given $m\geq 1$, Proposition~\ref{KoProp2.4} asserts that the trivial
  relations~\eqref{eq:triv} form, after tensoring by 
 $\Qq$, the submodule 
 ${\bf 1}\oplus G(M)$ where $M$ is the set of inverse roots of $L_{\rm red}((\Sym^m E),T)$. Hence 
 we see that linear independence for the inverse roots will 
 follow from the maximality of the Galois group of the splitting field of $L_{\rm red}((\Sym^m E),T)$ 
 over $\Qq$.
 
 More generally the implication we will use to prove our main results is the following. 
 If the Galois group of the splitting field over $\Qq$ of an elliptic curve $L$-function of the 
 type we consider is ``as big as possible'' (i.e.~contains $W_{2g}^+$ where $2g$ is the degree 
 of the associated reduced $L$-function) then this $L$-function will exhibit no nontrivial 
 multiplicative relations among its inverse roots. To give a first illustration 
 of this argument let us prove Proposition~\ref{KoProp1.1}.
 
 \medskip
 Notation being as in Proposition~\ref{KoProp1.1} we use the following
  result~(\cite[Th.~4.3]{Jou09}) about maximality of the Galois group over $\Qq$ 
of the splitting field of $L_{\rm red}(E_f/K,T)$ where $E/K$ is a fixed elliptic curve
 (with non-constant $j$-invariant)
 and the polynomials 
$f$ are obtained by letting $c$ run over $U_{\tilde{f}}(\Fp_q)$.
For any $\Qq$-polynomial $f$ let $\Gal_{\Qq}f$ be the Galois group of the splitting field 
of $f$ over $\Qq$.

\begin{theorem}[\cite{Jou09}] \label{th4.3}
 With notation as in~\S\ref{section:quadtwists} fix an elliptic curve $E/K$ 
 an integer $d\geq 2$
 and a polynomial $\tilde{f}\in\mathcal F_{d-1}(\Fp_q)$. For any $c\in U_{\tilde{f}}(\overline{\Fp_q})$
  let $E_c/K$ (resp. $L_{{\rm red},c}$, $N_{\rm red}$) be the quadratic twist of 
$E$ by $f(t)=(c-t)\tilde{f}(t)$ (resp. its reduced $L$-function, the common degree to all the reduced 
$L$-functions $L_{{\rm red},c}$). If 
$N:= \deg L(E_c/K,T)\geq 5$ (an integer which does not depend on $c$ but only on $d$ and $q$),
 $d\geq d_0(E)$, $q\geq q_0(E)$, then one has:
$$
\#\{c\in U_{\tilde{f}}(\Fp_q): \Gal_{\Qq}(L_{{\rm red},c})\not\supset W_{N_{\rm red}}^+\}\ll 
N^2q^{1-\gamma^{-1}}\log q\,,
$$
where the implied constant depends only on $j(E)$ and on $\tilde{f}$,
 where $d_0(E)$ and $q_0(E)$ depend only on $E$, and where one can choose 
$2\gamma =7N^2-7N+4$.
\end{theorem}

Notice 
that $c\in {\Fp_q}\setminus U_{\tilde{f}}(\Fp_q)$ if and only if $c$ is a root of $\tilde{f}$ or a 
root of $m$ (see~\eqref{twisting-space}). An immediate consequence of Theorem~\ref{th4.3} is:
$$
\#\{c\in \Fp_q\colon c\not\in U_{\tilde{f}}(\Fp_q)\text{ or } 
 \Gal_{\Qq}(L_{{\rm red},c}) \not\supset W_{N_{\rm red}^+}\}\ll N^2q^{1-\gamma^{-1}}\log q\,,
$$
with the same dependencies on the implied constant as in Theorem~\ref{th4.3}. 

 Proposition~\ref{KoProp1.1} then follows. Indeed any $c\in \Fp_q$ outside of the set on the left hand side of the inequality corresponds to 
 a $\Qq$-polynomial $L_{{\rm red},c}$ with a $\Zz$-multiplicatively independent set of zeros. 
 To see this fix such a $c\in\Fp_q$ and  apply Proposition~\ref{KoProp2.4} to $k=1$ 
 and $P=L_{{\rm red},c}$ (formally one should rather choose $P=T^{N_{\rm red}}L_{{\rm red},c}(1/T)$ 
 so that the roots are not confused with their inverses, however the set of zeros of $L_{{\rm red},c}$ 
 is stable under inversion). The fact that $\Gal_{\Qq}(L_{{\rm red},c})\simeq \mathcal W_{N_{\rm red}}$ concludes the proof.
 
 \begin{remark}
 We draw the reader's attention to the uniformity aspects of the inequality in 
 Proposition~\ref{KoProp1.1}. Notably we have a control on the dependency on the 
 common degree $N$ of the $L$-functions considered that we do not claim to obtain 
 in the statement of Theorem~\ref{th:main1}. This comes from the fact 
 that the proof of Proposition~\ref{KoProp1.1} relies on Theorem~\ref{th4.3} that builds 
 in turn on a Theorem of Hall~(\cite[Th.~6.3 and Th.~6.4]{Hal08}) where these uniformity 
 issues are handled with care whereas our proof of Theorem~\ref{th:main1} appeals 
 to Strong Approximation where one loses the effectiveness required to keep 
 track of the dependency on the degree of the $L$-functions.
 \end{remark}
  
   As is certainly clear from the way we have proven Proposition~\ref{KoProp1.1} we will 
   deduce our main results from maximality of Galois groups statements generalizing 
   Theorem~\ref{th4.3} (that will have to be adapted to the family of elliptic curves 
   introduced in Section~\ref{section:pullback}). 
   This will be done via a sieving procedure (generalizing the one developed 
   to prove~\cite[Th.~4.3]{Jou09}). A crucial input will be big $\ell$-adic monodromy 
   statements holding both for the families of Section~\ref{section:quadtwists} 
   and Section~\ref{section:pullback}.

\section{Some Large Sieve statements}\label{section:LS}

 We appeal to Kowalski's sieve for Frobenius. This technique is extensively described 
 in~\cite[Chapter $8$]{Kow08a}. Refinements of it are developed and used 
 in~\cite{Jou09}. For the purpose of 
 the present work an even more general sieve statement is needed. This comes first from the fact 
 that several polynomials 
  are to be considered at once (we are interested in the product of finitely many symmetric power 
  $L$-functions of a given elliptic curve) rather than just one (as is the case in~\cite{Jou09}).
  
  We first give a general sieve statement without specifying the property we investigate 
  (i.e.~a statement that holds for any choice of sieving sets in the language of~\cite{Kow08a}).

\begin{theorem}\label{sieve-statement}
 Let $\Fp_q$ be a finite field of $q$ elements and characteristic $p$. Let $V/\Fp_q$ be a smooth 
affine geometrically connected $d$-dimensional variety. Let 
$\kappa: V^{\rm cov}\rightarrow V$ be a Galois \'etale cover
with group $\mathcal G_V$ an elementary $2$-group.
 Assume further we are given a set of primes $\Lambda$ 
of positive density  that does not contain $p$ 
 such that 
for each $\ell\in\Lambda$,  we are given a lisse sheaf $\mathcal T_{d,\ell}$ (of rank denoted $r(d)$) of 
$\Fp_\ell$-vector spaces on $V$ corresponding to a homomorphism:
$$
\rho_\ell: \pi_1(V,\bar{\eta})\rightarrow GL(r(d),\Fp_\ell)\,,
$$
 that can be pulled back to a lisse  
sheaf $\kappa^* \mathcal T_{d,\ell}$ on $V^{\rm cov}$. We still denote by $\rho_\ell$ 
the corresponding representation:
$$
\rho_\ell\colon\pi_1(\overline{V^{\rm cov}},\bar{\mu})\rightarrow GL(r(d),\Fp_\ell)\,,
$$
where $\bar\mu$ is a geometric generic point that $\kappa$ maps to $\bar\eta$.
Set $G_\ell:=\rho_\ell(\pi_1(V,\bar{\eta}))$, 
$G_\ell^{\rm geom}:=\rho_\ell(\pi_1(\overline{V},\bar{\eta}))$ and 
$G_\ell^{{\rm geom},\rm cov}:=\rho_\ell(\pi_1(\overline{V^{\rm cov}},\bar{\mu}))$ and assume
\begin{itemize}
\item the product map
$$
\rho_{\ell,\ell'}: \pi_1(\overline{V^{\rm cov}},\bar{\mu})\rightarrow G_{\ell,\ell'}^{{\rm geom},\rm cov}:=G_\ell^{{\rm geom},\rm cov}\times G_{\ell'}^{{\rm geom},\rm cov}
$$
is onto for each $\ell\neq\ell'\in\Lambda$ (if $\ell=\ell'$ we define $\rho_{\ell,\ell'}:=\rho_{\ell}$),
\item for every $\ell\in\Lambda$ one has $p\nmid \# G_\ell^{{\rm geom},\rm cov}$.
\end{itemize}
 Let $\gamma_0$ be a representative of an element of the abelian 
 quotient $G_\ell/G_\ell^{{\rm geom}}$ (which corresponds to a union of left cosets 
 relative to $G_\ell^{{\rm geom},\rm cov}$)
 such that all the Frobenius conjugacy classes $\Frob_t$, $t\in V(\Fp_q)$
 map to $\gamma_0$ under $\rho_\ell$. Then for any choice of family 
 (indexed by $\Lambda$) of conjugacy invariant subsets
 $\Theta_\ell$ of the left coset $\gamma_0 G_\ell^{\rm geom}$ and any $L\geq 2$, one has:
\begin{equation}\label{eq:sieve}
\#\{t\in V(\Fp_q): \rho_\ell(\Frob_t)\not\in\Theta_\ell \text{ for all }\ell\leq L, \ell\in\Lambda
\}\leq \# \mathcal G_V(q^d+Cq^{d-1/2}(L+1)^A)(\delta(\Lambda)H)^{-1}\,
\end{equation}
where $\delta(\Lambda)$ is the density of $\Lambda$,
$$
H=\sum_{\substack{\ell\leq L\\ \ell\in\Lambda}}\frac{\#\Theta_\ell}{\#G_\ell^{\rm geom}-\#\Theta_\ell}\,,
$$
$C$ is a constant depending only on $\overline{V}$, and $A=7d'/2+1$ 
where $d'$ is the dimension of a connected component 
of maximal dimension 
 of the algebraic group underlying the $G_\ell$'s (i.e.~the algebraic group ${\bf G}/\Fp_\ell$ 
of minimal dimension such that $G_\ell\subseteq {\bf G}(\Fp_\ell)$).
\end{theorem}

\begin{remark}
The assumption that the Galois group $\mathcal G_V$ is an elementary $2$-group is not used in the proof. The reason we leave it as an assumption in the statement is because that condition holds in the context of our study of $L$-functions. Precisely the group $\mathcal G_V$ comes from a product of maximal abelian quotients of orthogonal groups over finite fields.
\end{remark}

\begin{proof}[Proof of Theorem~\ref{sieve-statement}]
First, the sieve statement has to be refined (or restricted) so that only those $t$'s in $V(\Fp_q)$
  such that $\Frob_t$ 
lies in a particular coset of $\pi_1(V,\bar{\eta})$ with respect to $\pi_1(V^{\rm cov},\bar{\eta})$ 
are considered. 
One needs first 
to fix an element $\alpha\in \mathcal G_V$ and sieve for the corresponding Frobenius conjugacy
 classes. Precisely, 
with notation as in the theorem set 
$$
X_\alpha:=\{t\in V(\Fp_q)\colon \tilde{\kappa}(\Frob_t)\in\alpha\}\,,
$$
where $\tilde{\kappa}: \pi_1(V,\bar{\eta})\rightarrow \mathcal G_V$ maps $\Frob_t$ 
to the action of $\pi_1(V,\bar{\eta})$ on $\kappa^{-1}(t)$. 
Then we claim
$$
\#\{t\in X_\alpha\colon \rho_\ell(\Frob_t)\not\in\tilde{\Theta}_\ell 
\text{ for all }\ell\leq L,\, \ell\in\Lambda\}
\leq (q^d+Cq^{d-1/2}(L+1)^A)(\delta(\Lambda)\tilde{H})^{-1}\,
$$
with the same notation and dependencies as in the theorem and where 
$\tilde{\Theta}_\ell$ is a conjugacy invariant subset of the coset of $G_\ell$ 
with respect to $G_\ell^{{\rm geom},{\rm cov}}$ (the quantity $\tilde{H}$ is defined 
the same way as $H$ up to replacing $\Theta_\ell$ (resp.~$G_\ell^{\rm geom}$) by $\tilde{\Theta}_\ell$ 
(resp.~$G_\ell^{{\rm geom},\rm cov}$)). To prove the claim we show that we can apply 
the coset sieve of~\cite[\S $3.3$]{Kow08a} with adjustments as in~\cite{Jou09}. 
Note that in both these papers $\Lambda$ is a set containing all but finitely 
many primes however it is straightforward to adapt the method to a set of primes of positive 
density $\delta(\Lambda)$.
This method is probably best described by considering the commutative diagram

\begin{equation} \label{diagcosetsieve}
 \begin{CD} 
 1 @>>> \pi_1(\overline{V^{\rm cov}},\bar{\mu}) @>>> \pi_1(V,\bar{\eta}) @>(\deg,\tilde{\kappa})>> \hat{\Zz}\times \mathcal G_V @>>> 1\\
@. @VV \rho_\ell V @VV \rho_\ell V @VV{\rm pr}_\ell V\\
1 @>>> G_{\ell}^{{\rm geom},\rm cov} @>>> G_\ell @>>> \Gamma_\ell @>>> 1\,,
\end{CD}
 \end{equation}
 where  ${\rm pr}_\ell$ (resp.~$\Gamma_\ell$) is the group morphism (resp.~the quotient 
 group) that makes the diagram commute. 
 
 In the terminology of~\cite{Kow08a} the coset sieve setting we use 
 is the triple $((-1,\alpha),\Lambda,(\rho_\ell))$ 
 where we see $(-1,\alpha)$ as a coset of $\pi_1(V,\bar{\eta})$ with respect to 
 $\pi_1(\overline{V^{\rm cov}},\bar{\mu})$. The sifted set (again in the sense of~\cite{Kow08a})
  attached is 
 $(X_\alpha,\text{counting measure}, \Frob)$ where $\Frob$ is the map 
 from closed points of $V$ to conjugacy classes of $\pi_1(V,\bar{\eta})$ mapping 
 $t$ to $\Frob_t$. The claim together with the upper bound~\eqref{eq:sieve} then follows 
 by applying~\cite[Cor. $3.7$ case $(1)$]{Jou09}. Note that in \emph{loc.~cit.~}we assume the algebraic group underlying the $G_\ell$'s is an orthogonal group. 
 However what we really need is merely an inequality of type
 $$
 \# U_0(\Fp_\ell)\leq (\ell+1)^\delta\,,
 $$
 where $U_0$ is a connected $\delta$-dimensional variety over $\Fp_\ell$. This 
 is a result due to Serre and we apply it to each connected component of the  
 algebraic group underlying $G_\ell$.

 \end{proof}
 
 We deduce a large sieve estimate involving polynomials of the type 
 we are investigating. In other words we show we can 
 apply Theorem~\ref{sieve-statement} to the concrete case where the property 
 studied is the maximality of the Galois group within a particular family of (characteristic) 
 polynomials. This amounts to specifying the sieving sets $\Theta_\ell$ appearing 
 in the statement of Theorem~\ref{sieve-statement}. Moreover we restrict 
 to finite groups $G_\ell$'s with underlying algebraic group a product 
 of orthogonal and symplectic groups since this will be the case in the applications 
 we have in mind.
 
 Let us briefly recall some useful facts about orthogonal groups over finite fields of characteristic 
 not $2$. 
 Let $\Oo(N,\Fp_\ell)$ be the group of isometries with respect to a non-degenerate 
 symmetric bilinear pairing $\Psi$ on an $N$-dimensional $\Fp_\ell$-vector space $V$. 
 The derived group $\Omega(N,\Fp_\ell)$ of $\Oo(N,\Fp_\ell)$ is 
 the simultaneous kernel of the determinant and of the spinor norm (the group 
 morphism from $\Oo(N,\Fp_\ell)$ to the group of classes of $\Fp_\ell^\times$
 modulo squares mapping a reflection $r_v$ with respect to the orthogonal space 
 of a non-isotropic vector $v$ to $\Psi(v,v)$). This group has index $2$ in ${\rm SO}(N,\Fp_\ell)$.
 In the even dimensional case the order of the orthogonal group depends 
 on the class modulo squares of the discriminant of the underlying quadratic form. Specifically 
 (see e.g.~\cite[Table $2.1$C]{KL90}):
 $$
 \# \Oo(N,\Fp_\ell)=\begin{cases}2\ell^{(\frac{N-1}{2})^2}\prod_{i=1}^{(N-1)/2}(\ell^{2i}-1)
& \text{ if $N$ is odd, }\\
2\ell^{\frac{N(N-2)}{4}}(\ell^{N/2}\mp 1)\prod_{i=1}^{N/2-1}(\ell^{2i}-1)
& \text{ if $N$ is even and $\left(\frac{\disc\, \Psi}{\ell}\right)=\pm 1$,}
 \end{cases}
 $$
 where $(\frac{\cdot}{\ell})$ denotes the Legendre character modulo $\ell$. 
In the even dimensional case distinct orders for orthogonal groups 
correspond either to a split (i.e.~$(-1)^{N/2}\det \Psi$ 
is a square) quadratic form or to a non split $\Fp_\ell$-quadratic space. 
 One easily deduces the existence 
 of positive constants $c_1(N),c_2(N)$ depending only on $N$ such that
 \begin{equation}\label{eq:orderorthogonal}
 c_1(N)\leq\frac{ \# \Oo(N,\Fp_\ell)}{\ell^{\frac{N(N-1)}{2}}}\leq c_2(N)\,,
 \end{equation}
 independently of the parity of $N$ and the class of the discriminant of $\Psi$ modulo squares.

 \par\medskip
 
 To state our next large sieve estimate we introduce some further notation and definitions. 
 Generalizing~\eqref{red-char} to any polynomial $f\in \Qq[T]$ of degree $N$ 
 satisfying an equation of the type
 \begin{equation}\label{eq:f}
 T^N f(1/T)=\eps(f)f(T)\,,\qquad \eps(f)=\pm 1\,,
 \end{equation}
 we define
$$
 f_{\rm red}(T)=\begin{cases} f(T)/(1+\eps(f)T)\,&\,\text{if}\,\,N\,\text{is odd}\,,\\
                                f(T)/(1-T^2)\,&\,\text{if}\,\,N\,\text{is even and}\,\,\eps(f)=-1\,,\\
                                 f(T)\,&\,\text{otherwise}\,.
                 \end{cases}
$$
 Finally let $k\geq 1$ be an integer and let $\mathcal F$  be a lisse $\Zz_\ell$-adic 
 sheaf on a $d$-dimensional 
 variety $V/\Fp_q$ whose
 arithmetic monodromy group modulo $\ell$ embeds in a product of type
 $$
  \prod_{1\leq m\leq k} {\bf G}(r(d,m),\Fp_\ell)\,,
  $$
  where for any ring $A$
  $$
  {\bf G}(r(d,m),A):=\begin{cases} \Oo(r(d,m),A)& \text{ if } m\text{ is odd}\,,\\
  {\rm CSp}(r(d,m),A)&\text{ if } m\text{ is even}\,.\end{cases}
 $$
Here $r(d,m)$ denotes an integer depending only on $m$ and $d$ and ${\rm CSp}(r(d,m),A)$ is the group 
of symplectic similitudes of a non-degenerate $r(d,m)$-dimensional $A$-module. We say that $\mathcal F$ has 
 \emph{big geometric monodromy modulo $\ell$} if there is a
    Galois \'etale cover $V^{\rm cov}/V$, with group $\mathcal G_V$ an elementary $2$-group, 
    whose geometric monodromy group modulo $\ell$ 
     contains
      $$
      \prod_{1\leq m\leq  k} {\bf G}'(r(d,m),\Fp_\ell)\,,
      $$
where for any ring $A$
  $$
  {\bf G'}(r(d,m),A):=\begin{cases} \Omega(r(d,m),A)&\text{ if } m\text{ is odd}\,,\\
\Sp(r(d,m),A)\, &\text{ if } m\text{ is even}\,.\end{cases}
 $$
 
   \begin{theorem}\label{product-sieve-statement}
   Assumptions on $V/\Fp_q$ are the same as in the statement of Theorem~\ref{sieve-statement}. 
   We keep the notation as above. 
   Let $k\geq 1$ be an integer.
   Let $\Lambda_{d,k}$ be a set of primes of positive density and suppose the density depends 
   only on the dimension $d$ of $V$ and on $k$.
    Suppose further that for each $\ell\in \Lambda_{d,k}$ we are given a 
   sheaf $\tilde{\mathcal T}_{d,k,\ell}$ of free $\Zz_\ell$-modules corresponding to a representation 
   $$
   \tilde{\rho}_\ell: \pi_1(V,\bar{\eta})\rightarrow
    \prod_{m=1}^{k} {\bf G}(r(d,m),\Zz_\ell)\,.
   $$
   For $n\in\{1,\ldots,k\}$ let $\tilde{\mathcal T}_{d,\ell}^{(n)}$ be the sheaf (with associated representation 
   denoted $\tilde{\rho}_{\ell}^{(n)}$)
 corresponding to the composition of $\tilde{\rho}_\ell$ with projection onto the $n$-th factor 
 (a sheaf with orthogonal or symplectic symmetry depending on the parity of $n$) 
   and assume $(\tilde{\mathcal T}_{d,\ell}^{(n)})_{\ell\in \Lambda_{d,k}}$ forms a compatible 
   system of $\Zz_\ell$-sheaves. Then $(\tilde{\mathcal T}_{d,k,\ell})_{\ell\in\Lambda_{d,k}}$
    is a compatible system of $\Zz_\ell$-sheaves. Let $f\in V(\Fp_q)$ and 
   $$
   L_f:=\det(1-T\tilde{\rho}_\ell(\Frob_f))\in\Zz[T]\,.
   $$
   Assume the following conditions are fulfilled:
   \begin{itemize}
   \item[(i)] the system $(\tilde{\mathcal T}_{d,k,\ell})_{\ell\in \Lambda_{d,k}}$
    has big geometric monodromy modulo $\ell$ for all $\ell\in\Lambda_{d,k}$, and the corresponding 
    cover $V^{\rm cov}/V$ does not depend on $\ell\in \Lambda_{d,k}$,
    
\item[(ii)]  $p\nmid G_{\ell}^{g,{\rm cov}}$ for all $\ell\in \Lambda_{d,k}$,
\item[(iii)] for all $\ell\in \Lambda_{d,k}$ and all $m$ either $r(d,2m-1)$ is odd or the orthogonal 
group $\Oo(r(d,2m-1),\Fp_\ell)$ appearing corresponds to a \emph{split} quadratic form 
over $\Fp_\ell$.
\end{itemize}
 Then we have for any sufficiently large power $q:=p^n$ ($n$ has to be chosen bigger 
 than a constant depending only on $\overline{V}$):
\begin{equation} \label{theoretic-bound}
\#\{f\in V(\Fp_q): L_{f,\rm red}\text{ is reducible or } 
\Gal_{\Qq}(L_{f,\rm red})\text{ is not maximal} \}\ll_{d,k} q^{d-\gamma^{-1}}\log q\,,
\end{equation}
where one can choose:
$$
2\gamma=4+7\sum_{m=1}^{k} \tilde{h}(m)\,,\qquad \tilde{h}(m):=\begin{cases}
&r(d,m)\left(r(d,m)-1\right)\text{ if } m\text{ is odd}\,,\\
&r(d,m)\left(r(d,m)+1\right)\text{ if } m\text{ is even}\,. \end{cases}
$$ 
Here ``maximal'' means that the corresponding Galois group is isomorphic to $\mathcal W^{(k)}$ 
(with notation 
 as in Corollary~\ref{KoCor2.3}).

For any $1\leq j\leq k$ let  
$$
L_{f,j}=\det(1-T\tilde{\rho}_{\ell,j}(\Frob_f))\in\Zz[T]\,,
$$
then in the above estimate $L_{f,\rm red}$ denotes the product over 
even indices of $L_{f,j}$ times the product over odd indices of $L_{f,j,{\rm red}}$.
 The implied constant in the upper bound depends only on $d$ and $k$.
  \end{theorem}

\begin{proof}
First note that one has trivially:
$$
L_f=\prod_{1\leq m\leq k}\det(1-T\tilde{\rho}_{\ell}^{(m)}(\Frob_{f}))\,,
$$
so that $(\tilde{\mathcal T}_{d,k,\ell})_{\ell\in\Lambda_{d,k}}$ is automatically a 
compatible system of $\Zz_\ell$-sheaves. 

To prove~\eqref{theoretic-bound} we follow the strategy of~\cite[Proof of Th.~$4.3$]{Kow08b} 
where only sheaves exhibiting symplectic symmetry were needed. 
In \emph{loc.~cit.~}the author recalls that in earlier work of his he defined four sets 
$\Theta_{i,\ell}\subseteq {\rm CSp}(2g,\Fp_\ell)$, $1\leq i\leq 4$, that detect the maximality 
of the Galois group of the $\Qq$-polynomial investigated. Since~\cite[Th.~$4.3$]{Kow08b} 
deals, as we do, with \emph{products} of characteristic polynomials an additional sieving set 
(to which the index $i=0$ is attributed) is introduced in the proof to guarantee that 
the Galois group obtained does not 
merely surject onto each factor of the product group $\mathcal W^{(k)}$ but is in 
fact equal to the whole group $\mathcal W^{(k)}$.

Likewise four families of sieving sets were identified in~\cite{Jou09} (where only 
sheaves exhibiting orthogonal symmetry appeared) and shown 
to be sufficient to ensure maximality of the Galois group investigated. However we also 
need a suitable ``zeroth'' family $(\Theta_\ell^{(0)})$ (see Lemma~\ref{lem:count} for the 
definition)
of sieving sets to guarantee the maximality of  the Galois 
group as a product group. Because of complications with 
orthogonal groups one needs to be extra careful in our case 
 when handling multi-indices ${\bf i}=(i_1,\ldots, i_k)$ 
where $i_m=0$ for some odd $m$. This is the reason why we have to impose a particular value 
of the discriminant (modulo squares) of the quadratic spaces coming into play in the statement.
 Lemma~\ref{lem:count} (the proof of 
which we postpone till the end of the section) asserts that we do have a
lower bound on the density of sets $\Theta_\ell^{(0)}$ of type
$$
\#\Theta^{(0)}_\ell/\#\Omega(r(d,m),\Fp_\ell)\gg 1\,,
$$
for every odd $m$, with an implied constant depending only on $r(d,m)$. 

At even indices, the family of sieving sets $(\Theta_\ell)$ we choose is the same as 
in~\cite{Kow08b}. Now denoting $c_i^{(m)}$ 
the element (determined up to conjugation) of the Galois group of the $m$-th factor 
corresponding to the sieving set $\Theta_\ell^{(i)}$, for $1\leq i\leq 4$, and noticing
 that the trivial permutation of the Galois group corresponds to the sieving sets
  $\Theta^{(0)}_\ell$ one deduces 
that in the Galois group investigated one may use sieve to detect all permutations of type
$$
(1,\ldots,1,c_i^{(m)},1\ldots,1)\,,
$$
for any $1\leq i\leq 4$ and any $1\leq m\leq k$. If all these permutations are successfully
 detected we conclude that the Galois group is isomorphic to $\mathcal W^{(k)}$. 
In particular, we only need to consider the $4k$ families $(\Theta^{({\bf i})}_\ell)$ 
(for indices ${\bf i}$ as described above) for our purpose.

 To be in the context of Theorem~\ref{sieve-statement} it remains to check the linear 
 disjointness condition for product representations $\rho_{\ell,\ell'}$, for $\ell\neq \ell' 
\in\Lambda_{d,k}$. Kowalski's argument 
(\cite[Lemma $4.4$]{Kow08b}) can easily be generalized to our setting
 thanks to the group theoretical properties shared by the groups $\Sp(2n,\Fp_\ell)$ and 
$\Omega(n,\Fp_\ell)$: both are groups with all normal subgroups contained 
in the center. Let $N(q)$ be the left-hand side of~\eqref{theoretic-bound}. 
By the above considerations and applying Theorem~\ref{sieve-statement} we get the
inequality:
$$
N(q)\leq \#\mathcal G_V\cdot(4k)\cdot(q^d+Cq^{d-1/2}(L+1)^A)(\delta(\Lambda_{d,k})H)^{-1}\,,
$$ 
for any $L\geq \min \Lambda_{d,k}$ and where one can choose
$$
H=\min_{\bf i}\sum_{\ell\leq L}\left(\frac{\#\Theta_\ell^{({\bf i})}}{
\#\prod_{1\leq m\leq k} {\bf G}^{'}(r(d,m),\Fp_\ell)}\right)\,,
$$
and 
$$
A=1+7\sum_{m=1}^{k} \dim {\bf  H}(r(d,m))\,\qquad {\bf H}(r(d,m)):=\begin{cases}
&\Oo(r(d,m)\text{ if } m\text{ is odd}\,,\\ &\Sp((r(d,m))\text{ if } m \text{ is even}\,.\end{cases}
$$
Then we choose $L$ such that $CL^A=q^{1/2}$ that is $L=(qC^{-2})^{1/(2A)}$ 
(this quantity is greater than $\min\Lambda_{d,k}$ as long as $q$ is a big enough power of $p$ 
; the exponent depends only on the constant $C$ which in turn depends only on $\overline{V}$).
 Thus the choice $\gamma=2A$ 
is suitable and the upper bound stated follows from the well known formul\ae\, for the dimension 
of the orthogonal and symplectic groups.
\end{proof}

 We now state the following counterpart of Theorem~\ref{product-sieve-statement}
  in terms of independence of the zeros.

\begin{corollary}\label{cor4.7}
Keeping notation as in Theorem~\ref{product-sieve-statement} denote by 
$\mathcal Z(L_{f,{\rm red}})$ the (multi-) set of inverse roots of the reduced version of the polynomial
$L_f$. Then we have
$$
\#\left\{f\in V(\Fp_q): 
{\rm Rel}\left(\mathcal Z(L_{f,{\rm red}})\right)\text{ is nontrivial }\right\}\ll q^{d-\gamma^{-1}}\log q\,,
$$
where one can take 
$$
2\gamma=4+7\sum_{m=1}^{k} \tilde{h}(m)\,,
$$ 
 and where the implied constant depends only on $d$ and $k$.
\end{corollary}

\begin{proof}
 The argument is the same as the one used to deduce Proposition~\ref{KoProp1.1} 
 from Theorem~\ref{th4.3} (i.e.~the relationship between maximality of the Galois group and 
 independence of the zeros explained in \S\ref{section:implication}).
 The functional equation satisfied by each $L_{f,m,{\rm red}}$ is~\eqref{eq:f} in the case 
 where the degree is even and the sign of the functional equation is $+1$. Therefore the (multi-)set 
 of zeros of $L_{f,{\rm red}}$ coincides with the (multi-)set of its inverse zeros. 
 Proposition~\ref{KoProp2.4} states that as long as $\Gal_{\Qq}(L_{f,{\rm red}})$ is maximal 
 (i.e.~isomorphic to a group of type $\mathcal W^{(k)}$)
 then the module of relations among the zeros 
 of $L_{f,{\rm red}}$ is 
 $$
 \bigoplus_{1\leq j\leq k}\left({\bf 1}\oplus G(M_j)\right)\,,
 $$
 i.e.~it reduces to the relations imposed by the functional equation satisfied by $L_{f,{\rm red}}$ 
 (the so-called trivial relations~\eqref{eq:triv}).
\end{proof}

We end this section with the statement and the proof of the counting lemma 
needed in the proof of Theorem~\ref{product-sieve-statement}. For simplicity all congruences 
in the sequel will mean ``congruences modulo the group of non-zero squares of $\Fp_\ell$''.

\begin{lemma}\label{lem:count}
 Let $N\geq 4$ be an integer and $\ell\geq 3$ be a prime number.
  Let $f$ be a monic polynomial of degree $N$
  satisfying~\eqref{eq:f} then there is a non-degenerate $N$-dimensional quadratic 
   $\Fp_\ell$-space
   $(V,\Psi)$ and an 
   isometry $\gamma$ of this space such that $\det(T-\gamma)=f(T)$. Moreover if $N$ is even 
   and $f(\pm 1)\neq 0$ then
   one has necessarily $\det \Psi\equiv f(-1)f(1)$.
   
   If $N$ is even assume that the quadratic structure $(V,\Psi)$ 
   is \emph{split} i.e.~$(-1)^{N/2}\det \Psi\equiv 1$.
   
    Let $\Omega(N,\Fp_\ell)$ 
   be the derived group of the orthogonal group $\Oo(V)$ and let $\alpha_\ell$ 
   be a representative of the four classes of $\Oo(V)$ with respect to $\Omega(N,\Fp_\ell)$. 
   If we set
   $$
   \Theta_\ell^{(0)}:=\{M\in \alpha_\ell\Omega(N,\Fp_\ell): \det(1-TM)\text{ is separable and split over }
   \Fp_\ell\}
   $$
   then we have 
   $$
   \frac{\#\Theta_\ell^{(0)}}{\# \Omega(N,\Fp_\ell)}\gg_N 1\,.
   $$
\end{lemma}

\begin{proof}
The first part of the statement can be deduced from transfer arguments 
(see e.g.~\cite[Th.~$4.1$ and Prop.~$6.2$]{JRV14}). 

Let us turn to the proof of the lower bound for $\#\Theta_\ell^{(0)}/\# \Omega(N,\Fp_\ell)$.
 We first claim that we may assume without loss of generality that $\det \alpha_\ell=1$ and 
$N$ is even (i.e.~$N=N_{\rm red}$, where we recall that the characteristic polynomial of 
$\alpha_\ell$ satisfies~\eqref{eq:f} and where 
$N_{\rm red}$ is defined as the degree of its reduced version). 
Indeed we are only counting isometries that have a separable characteristic  
polynomial. If either $N$ is odd or the determinant of such an isometry $M$ is $-1$ the functional 
equation~\eqref{eq:f} will impose $\pm 1$ (or both) to be an eigenvalue of multiplicity 
one of $M$. The corresponding eigenspace $V_1$ (or $V_{-1}$, or both)
 has dimension $1$ and we have an orthogonal 
splitting (see e.g.~\cite[(6.3) and the references mentioned in the proof of Corollary $6.4$]{JRV14})
$$
V_{\pm 1}\bot V_{N_{\rm red}}\,,
$$
where $V_{\pm 1}$ stands either for $V_1$, $V_{-1}$ or the orthogonal sum of both, depending 
on the parity of $N$ and on the sign of $\det M$. The isometry 
$M$ restricts to the non-degenerate $N_{\rm red}$-dimensional 
quadratic space $V_{N_{\rm red}}$ as an isometry 
of determinant $1$. Up to imposing a split or non split 
quadratic structure on $V_{\pm 1}$ we can further assume 
$(-1)^{N/2}\det (V,\Psi)\equiv (-1)^{N_{\rm red}/2}\det (V_{N_{\rm red}},\Psi)$. This proves the claim. In particular in the rest of the proof 
we will use the fact that the set of roots and the set of reciprocal roots of the characteristic 
polynomials considered are the same.

The strategy is then to apply~\cite[Th.~$15$]{Jou10} i.e.~we reduce the question to that of 
counting candidate polynomials. Here these polynomials are reciprocal, monic, 
of degree $N$ and split over $\Fp_\ell$ with pairwise distinct roots. The reduction step from 
the general case to the case $N$ even and $\det=1$ imposes the roots of the candidate 
polynomials to be different from $\pm 1$. 

The structure of the candidate polynomials explains why we impose the split structure on the
orthogonal group. Indeed if $f$ is the characteristic polynomial (not vanishing 
at $\pm 1$) of an isometry of an even 
dimensional non-degenerate $\Fp_\ell$-quadratic space $(V,\Psi)$ one has 
(see~\cite[Prop.~$6.2$, Lemma $6.5$]{JRV14})
$$
\det \Psi\equiv (-1)^{N/2}\disc f\equiv f(1)f(-1)\,.
$$
Here we consider polynomials $f$ that are products of split quadratic 
polynomials of type $(T-\beta)(T-\beta^{-1})$. Note that
$$
(1-\beta)(1-\beta^{-1})(-1-\beta)(-1-\beta^{-1})=-(\beta-\beta^{-1})^2
$$
and hence $(-1)^{N/2}\det \Psi\equiv 1$ meaning $(V,\Psi)$ 
is a split quadratic $\Fp_\ell$-space.

To produce the candidate polynomials 
we split $\Fp_\ell^\times\setminus\{\pm 1\}$ in two disjoint subsets so that inversion 
induces a bijection between these two subsets. There are 
\begin{equation}\label{eq:binom}
\binom{\frac{\ell-3}{2}}{\frac{N}{2}}
\end{equation}
ways of picking $N/2$ suitable roots for the polynomials we consider (note each time 
we pick a root, its inverse will automatically be a root as well) all in the same one
 of the two subsets we have just described.
Since we are only interested in isometries with 
prescribed spinor norm (imposed by the choice of $\alpha_\ell$)
 we have yet to show that a positive 
proportion of the polynomials constructed correspond to an 
isometry of prescribed spinor norm. 
For that purpose we use the following result due to Zassenhaus (see e.g.~\cite[Th.~$5.1$]{JRV14} 
and the references therein). For any isometry $M$ 
of a quadratic (even dimensional) 
space $V$ that has a characteristic polynomial $f$ not vanishing at $\pm 1$
$$
 \nsp (M)\equiv f(-1)\,.
$$
Thus we want to check that the polynomials $f$ we have 
constructed take values $f(-1)$ that are roughly equidistributed in $\Fp_\ell^\times/\Fp_\ell^{\times 2}$. 
It is enough to show that this equidistribution property holds for any quadratic factor of degree $2$ 
of the polynomials we consider. Let
$(T-\beta)(T-\beta^{-1})$
be such a factor ($\beta\in\Fp_\ell^\times\setminus\{\pm 1\}$). Its value at $-1$ is 
$2+\beta+\beta^{-1}$ thus
$$
\left(\frac{2+\beta+\beta^{-1}}{\ell}\right)=
\left(\frac{\beta^2+2\beta+1}{\ell}\right)\left(\frac{\beta}{\ell}\right)=\left(\frac{\beta}{\ell}\right)\,.
$$
Thus using orthogonality relations we deduce
\begin{align*}
\#\{\beta\in\Fp_\ell^\times\setminus\{\pm 1\}\colon 2+\beta+\beta^{-1}\text{ is a square }\}
&=\frac{1}{2}\sum_{\beta\in\Fp_\ell^\times\setminus\{\pm 1\}}
\left(1+\left(\frac{2+\beta+\beta^{-1}}{\ell}\right)\right)\\
&=\frac{\ell-3}{2}+\frac{1}{2}\sum_{\beta\in\Fp_\ell^\times\setminus\{\pm 1\}}\left(\frac{\beta}{\ell}\right)
=\frac{\ell-3}{2}+O(1)\,,
\end{align*}
with an absolute implied constant.

Using~\eqref{eq:binom}, the above equidistribution fact and the classical lower bound 
on binomial coefficients $\binom{n}{k}\geq (n/k)^k$ we deduce
$$
\#\{f\in \Fp_\ell[T]\colon \deg f=N\,, f\text{ reciprocal, split, separable and } f(-1)\equiv \nsp(\alpha_\ell)\}
\gg_N \ell^{N/2}\,.
$$
The lower bound stated in the lemma follows from the above 
lower bound combined with~\eqref{eq:orderorthogonal} and~\cite[Th.~$15$]{Jou10}.
\end{proof}

 Without much extra work one could keep track along the proof  
 of the dependency on $N$ and thus get a uniform version of the lower bound of 
 Lemma~\ref{lem:count}. This was done for quite general sieving sets 
 in~\cite[Lemma $16$]{Jou10}. However for the application we have in mind in the present paper 
 the qualitative upper bound of Lemma~\ref{lem:count} suffices and for simplicity 
 we have chosen not to include the extra details that would lead to a uniform lower bound.


\section{Proof of the main results}\label{section:proof}

 In this section we prove Theorem~\ref{th:main1} and Theorem~\ref{th:main2}.
  We first explain the cohomological genesis 
 of the $L$-functions we study (i.e.~$L$-functions for families of elliptic curves described in 
 \S\ref{section:quadtwists} and \S\ref{section:pullback}). As already mentioned both these 
 families enjoy the property of having big geometric $\ell$-adic monodromy. Then we explain 
 how one deduces big monodromy modulo $\ell$ (for a big enough set 
 of primes) for these families 
 from the corresponding $\ell$-adic information. Combining these ingredients all the 
 assumptions needed for Theorem~\ref{product-sieve-statement} to apply will be 
 satisfied.

\subsection{Cohomological genesis of the $L$-functions considered}\label{section:cohgen}

 We describe very briefly the constructions of Katz leading to the two families 
 of elliptic curve $L$-functions we study.
 
 \subsubsection{The quadratic twist family}

We first focus on the family of quadratic twist $L$-functions of \S\ref{section:quadtwists}.

As before let $\ell$ be a rational prime invertible in $\Fp_q$ and let $E/K$ be an elliptic curve
 over $K=\Fp_q(C)$ with non-constant $j$-invariant and minimal Weierstrass  model 
  $\mathcal E\rightarrow C$. 
  There is an open dense curve with corresponding inclusion $j: U\subset C$ such 
  that each fiber of $\varpi: \mathcal E\rightarrow U$ is an elliptic curve.
  
   On $U$ we consider the constant $\ell$-adic sheaf $\overline{\Qq_\ell}$. 
   The sheaf ${\rm R}^1\varpi_\star  \overline{\Qq_\ell}$ on $U$ built out of
    the constant $\ell$-adic sheaf and of $\varpi$, is lisse of rank two,
    pure of weight one and everywhere tame if $p:={\rm char}\, \Fp_q\geq 5$ 
    (which is indeed the case throughout the paper by assumption). 
    The Tate twist ${\rm R}^1\varpi_\star  \overline{\Qq_\ell}(1/2)$ of that 
     sheaf is therefore of rank two, pure of weight zero, and symplectically self-dual (because 
     of the Weil pairing on the elliptic curve $E/K$). Define the sheaf
     $$
     \mathcal S:=j_\star{\rm R}^1\varpi_\star\overline{\Qq_\ell}(1/2)
     $$
     on $\Pp^1$. The open set on which $\mathcal S$ is lisse is the largest
     open set over which $E/K$ has good reduction. Given $n\geqslant 1$ 
     one can consider the symmetric $n$-th power of 
     ${\rm R}^1\varpi_\star\overline{\Qq_\ell}(1/2)$ on $U$ 
     (since this sheaf corresponds to a continous $\ell$-adic representation of
     the \'etale fundamental group of $U$ (with respect to a fixed base point)). This 
     sheaf $\Sym^n{\rm R}^1\varpi_\star\overline{\Qq_\ell}(1/2)$ is lisse on $U$ of rank $n+1$, 
     pure of weight zero, and everywhere tame. It is symplectically (resp.~orthogonally) self-dual 
     if $n$ is odd (resp.~if $n$ is even). Using the inclusion $j$, one can then define
     $$
     \mathcal S_n:=j_\star\Sym^n{\rm R}^1\varpi_\star\overline{\Qq_\ell}(1/2)\,,
     $$
which is a geometrically irreducible middle-extension sheaf on $\Pp_1$ (this comes from 
the fact, proven in~\cite[\S$3.5.5$]{Del80}, 
that  ${\rm R}^1\varpi_\star\overline{\Qq_\ell}(1/2)$ has $SL_2$ geometric monodromy).

 The above sheaf-theoretic constructions can be combined with twisting operations. 
 By a recipe described 
 by Katz in~\cite[\S $5.2.1$]{Kat02}, there is a lisse $\ell$-adic sheaf $\mathcal T_{d,n}$
  on $\mathcal F_d$ 
 (the singular locus of $\mathcal S_n$ being contained in the singular locus of $\mathcal S$ for any $n\geq 1$ )
 whose stalk at $f\in \mathcal F_d$ is 
 $H^1(\Pp_1,j_\star(\mathcal S_n\otimes \mathcal L_{\chi(f)}))$ (note 
 that one might have to slightly modify what the inclusion $j$ is so that the resulting sheaf is lisse). Here 
 $\mathcal L_\chi$ denotes the Lang sheaf associated to the Legendre character $\chi$ of $\Fp_q$ and 
 $\mathcal L_{\chi(f)}:=f^\star\mathcal L_\chi$. The key property we need is the following 
 ``big $\ell$-adic monodromy'' statement (see~\cite[Th.~7.6.7]{Kat05}).
 
 \begin{theorem}[Katz]\label{KatzBigMonodromy}
  With notation as above, let $N_{d,n}$ denote the rank of the $\ell$-adic sheaf $\mathcal T_{d,n}$.
  \begin{enumerate}
  \item If $n$ is even then the lisse sheaf $\mathcal T_{d,n}(1/2)$ on 
  $\mathcal F_d$ is pure of weight zero 
  and symplectically self-dual with geometric monodromy 
  group $\Sp(N_{d,n})$,
  
  \item if $n$ is odd and $E$ has multiplicative reduction at at least one closed 
  point $\pi\in \Pp^1(\overline{\Fp_q})$,
 then the lisse sheaf $\mathcal T_{d,n}(1/2)$ on $\mathcal F_d$ is pure of weight zero 
  and orthogonally self-dual with geometric monodromy group $\Oo(N_{d,n})$. 
  \end{enumerate}
  Fix an embedding $\iota: \overline{\Qq_\ell}\hookrightarrow \Cc$; for each finite extension 
  $\Fp/\Fp_q$ 
  and each $f\in\mathcal F_d(\Fp)$, let $\Theta_{\Fp,f}$ be the Frobenius conjugacy class in 
  ${\rm USp}(N_{d,n})$ (resp.~in $\Oo(N_{d,n},\Rr)$)
 corresponding to $\mathcal T_{d,n}(1/2)$ if $N_{d,n}$ is even (resp.~odd) at 
 $f\in \mathcal F_d(\Fp)$. Then
  $$
L\left((\Sym^n \rho_{\ell,E/K})\otimes \chi_f,T\right)=\det(1-\Theta_{\Fp,f}T)=
  \iota\left(\det\left(1-T\Frob_{\Fp,f}\mid \mathcal T_{d,n}(1/2)\right)\right)\,,
  $$
  where we recall that $\chi_f$ is the unique nontrivial $K$-automorphism of $K(\sqrt{f})$.
 \end{theorem}

Let us comment on the last sentence of the statement. The $\Qq$-polynomial 
$L\left((\Sym^n \rho_{\ell,E/K})\otimes\chi_f,T\right)$
 \emph{does not} coincide a priori with the symmetric power $L$-function 
of the representation giving rise to $L(E_f/K,T)$. More precisely the operations 
of ``twisting'' and ``taking the $n$-th symmetric power'' do not commute in general
as the following lemma shows.

\begin{lemma}\label{lem:operationscommute}
With notation as in Theorem~\ref{KatzBigMonodromy} one has for every integer $n\geq 1$,
$$
L((\Sym^n E_f)/K,T)=\begin{cases} & 
L\left((\Sym^n \rho_{\ell,E/K})\otimes\chi_f,T\right)\text{ if $n$ is odd, }\\
& \left(L((\Sym^{n} E)/K,T)\right)  \text{ if $n$ is even. }\end{cases}
$$
\end{lemma} 

\begin{proof}
We want to compare the $L$-functions of the representations
 $$
 \Sym^n\left(\rho_{\ell,E/K}\otimes \chi_f\right)
 \text{ and }
 \left(\Sym^n \rho_{\ell,E/K}\right)\otimes\chi_f\,.
  $$
  In each case the unramified places are the places of good reduction of $E/K$ 
  that do not correspond to an irreducible factor of $f$. Let $v$ be a common unramified place for 
  the two $L$-functions we consider.
Combining~\eqref{eq:LocalSym} and the straightforward generalization of~\eqref{eq:local-twist}  
to all symmetric powers $\Sym^n \rho_{\ell,E/K}$ we see that 
the local factor at $v$ of the $2m$-th (resp.~$(2m+1)$-th)
 symmetric power of the quadratic twist of $E$ by $f$ is exactly 
the same as the $2m$-th (resp.~$(2m+1)$-th) symmetric power of the original curve $E/K$ 
(resp. of the twist $E_f/K$).

By Chebotarev's Density Theorem it is enough to check the local factors of both $L$-functions
coincide at all 
unramified places to deduce that the $L$-functions are the same (or indeed that the underlying 
representations of the \'etale fundamental group of a maximal open subset on which 
they both are unramified are isomorphic). In the context of $L$-functions of elliptic curves 
over function fields this type of argument is used e.g.~in~\cite[Rem.~$7.0.5$]{Kat02}.
\end{proof}



\begin{remark}\label{rem:notwist}
(i) For our quadratic twist family of $L$-functions the lemma explains why we can only 
hope for the simultaneous independence of zeros when taking (a finite number of) \emph{odd} 
symmetric power $L$-functions. 

(ii) For sieving purposes it will be convenient in the case where $n$ is even 
  (i.e.~$\mathcal T_{d,n}$ is symplectically self-dual), \emph{not} 
to perform the ``half Tate twist'' as described in the statement of Theorem~\ref{KatzBigMonodromy}.
 The reason is that it is convenient to have 
an arithmetic monodromy group 
that embeds in the symplectic \emph{similitudes} ${\rm CSp}(2g)$ so that we can choose the 
multiplicator of the similitudes as a sieving parameter. 
\end{remark}

\subsubsection{The pullback family}

Let us now turn to the interpretation of  $L^{\rm new}((\Sym^n E^f)/\Fp_q(C),T)$ (seen as the  $\Qq$-polynomial  
defined in \S\ref{section:pullback} starting with the elliptic curve given by~\eqref{eq:gen-weierstrass}) 
as being the characteristic polynomial of the (global) geometric Frobenius morphism acting on an $\ell$-adic 
cohomology space. The construction once again is due to Katz. Our exposition
 follows closely~\cite[\S $7.3$]{Kat05} in which much more details (together with 
 full proofs and other applications) are given.
 
 As an auxiliary piece of data we fix an effective divisor $D$ on $C/\Fp_q$ satisfying 
 $\deg D\geq 2g+3$ where we recall that $g$ is the genus of $C/\Fp_q$. Let $S\subseteq {\bf A}^1$ 
 be the locus of bad reduction of the curve $E$ given by~\eqref{eq:gen-weierstrass}. Similarily to 
 the case of the other family considered we assume $S$ contains at least one place of multiplicative reduction and that $E/\Fp_q(t)$ has non-constant $j$-invariant
  (see~\cite[$(7.3.2)$]{Kat05} where Katz explicitly makes these assumptions). 
 We let 
 $U_{D,S}$ be the dense open subset of the Riemann--Roch space $\mathcal L(D)$ whose 
 $\overline{\Fp_q}$-valued points consists of those $f\in\mathcal L(D)/\overline{\Fp_q}$ 
 whose divisor of poles is $D$ and which are finite \'etale over $S$. In~\cite[\S$7.3.12$]{Kat05} 
 it is stated that for any $n\geq 1$ there is a lisse $\overline{\Qq}_\ell$-sheaf $\mathcal M_n$
  on $U_{D,S}$ (the fact that we assumed 
 that $p\geq 5$ plays a role here) such that for any finite extension $\Fp/\Fp_q$ and any 
 $f\in U_{D,S}(\Fp)$ one has
 $$
 L^{\rm new}((\Sym^n E^f)/\Fp_q(C),T)=\det\left(1-T\Frob_{\Fp,f}\mid \mathcal M_n\right)\,.
 $$
 Moreover one has the following big monodromy statement 
 (see~\cite[Th.~$7.3.14$, $7.3.16$]{Kat05}) of the same type as 
 Theorem~\ref{KatzBigMonodromy}.
 
 \begin{theorem}[Katz]\label{KBM2}
 Let $N_n$ be the rank of the sheaf $\mathcal M_n$.
 
For any $n\geq 2$ the geometric monodromy group of $\mathcal M_n$ is $\Oo(N_n)$ if 
$n$ is odd and ${\rm Sp}(N_n)$ if $n$ is even. In both cases for any finite extension 
$\Fp/\Fp_q$ and any $f\in U_{D,S}(\Fp)$ the global 
geometric Frobenius $\Frob_{\Fp,f}$ acts as an isometry with respect to the 
associated bilinear structure.

Assuming further that $N_1\geq 9$ the geometric monodromy 
group of $\mathcal M_1$ is $\Oo(N_1)$.
\end{theorem}

\subsection{Big monodromy modulo $\ell$}

 Our sieve setting imposes knowledge of the reduction modulo $\ell$ 
 of the $L$-functions we consider modulo many primes $\ell$. In this section we state 
 a result of
  big monodromy modulo $\ell$ analogous to (and deduced from) 
 Theorem~\ref{KatzBigMonodromy} and Theorem~\ref{KBM2}. The other ingredient 
 is the celebrated Strong Approximation Theorem~\cite{MVW84} of
 Matthews, Vaserstein and Weisfeiler, enabling one to obtain 
 big monodromy modulo $\ell$ for all but finitely many primes $\ell$.
 An alternative method would consist in exploiting a Theorem of Larsen~\cite{Lar95} 
 that would produce a set of ``good primes'' of density $1$.

Both these methods (Strong Approximation and Larsen's argument) are
 explained in detail in~\cite[\S$7$ and \S$9$]{Kat12}. Here we merely 
quote Katz's argument and refer the reader to \emph{loc.~cit.~}for the details.

Once more the argument is simpler for sheaves exhibiting  symplectic symmetry 
as opposed to sheaves with orthogonal symmetry. The reason is topological: the symplectic 
group ${\rm Sp}(2g)$ is a simply connected algebraic group while neither 
$\Oo(N)$ nor its connected component ${\rm SO}(N)$ are. Thus while Stong Approximation 
may be applied directly to a Zariski dense subgroup in the former case one has to go 
to the simply connected cover ${\rm Spin}(N)$ of ${\rm SO}(N)$ first in the latter case.

Let $\mathcal H_n$ (resp. $U$) be either of the sheaves $\mathcal T_{d,n}$ 
(resp. the parameter variety  $\mathcal F_d$)
of Theorem~\ref{KatzBigMonodromy} or $\mathcal M_n$ (resp. the parameter 
variety $U_{D,S}$) of Theorem~\ref{KBM2} . Let 
$\G/\Zz$ be either of the groups $\Sp({\rm rk}\, \mathcal H_n)$ if $n$ is even or 
$\Oo({\rm rk}\, \mathcal H_n)$ if $n$ is odd.
Katz explains in~\cite[\S $9$ and proof of Th.~$3.1$]{Kat12} that there exists an integer 
$N_0\geq 1$ and a sheaf $\mathcal H_{\Zz[1/{N_0]}}$ of $\Zz[1/{N_0}]$-modules such that for 
$\ell\nmid N_0$
the $\ell$-adic geometric monodromy of $\mathcal H_n$ is the $\ell$-adic closure in 
$\G(\Zz_\ell)$ of a finitely generated Zariski-dense 
subgroup $\Gamma_{N_0}\subseteq\G(\Zz[1/{N_0}])$. (In \emph{loc.~cit.~}Katz only considers the case where $\mathcal H_n$ has orthogonal
 symmetry but the same argument works in the symplectic case.) We would like 
to apply Strong Approximation to $\Gamma_{N_0}$ but again this is only 
directly possible in case $\G=\Sp({\rm rk}\,\mathcal H_n)$. In~\cite[\S9]{Kat12} 
Katz explains a way (for which he acknowledges R. Livn\'e) to circumvent 
this difficulty by going to the spin double cover of ${\rm SO}({\rm rk}\,\mathcal H_n)$. 

Let us state the outcome of the above line of reasoning.

\begin{proposition}\label{prop:monodromymodell}
With notation as above let $\Gamma_{N_0, {\rm mod} \ell}$ denote the image in $\G(\Fp_\ell)$ 
 of the geometric monodromy group of $\mathcal H_n$. Then 
 this group is also the image by reduction modulo $\ell$ of the subgroup 
 $\Gamma_{N_0} \subseteq\G(\Zz[1/{N_0}])$. Moreover
\begin{itemize}
 \item if $\ell\nmid N_0$ and $n$ is even then 
 $\Gamma_{N_0, {\rm mod} \ell}=\Sp({\rm rk} \mathcal H,\Fp_\ell)$,
\item  if $\ell\nmid N_0$ and $n$ is odd then 
$\Gamma_{N_0, {\rm mod} \ell}
\supset \Omega({\rm rk} \mathcal H,\Fp_\ell)$ and the underlying quadratic 
form is obtained by reduction modulo $\ell$ of a quadratic form over $\Zz[1/N_0]$.
\end{itemize}
\end{proposition}

 The second part of the statement has to be made more explicit. To apply 
 Theorem~\ref{product-sieve-statement} we need to have a control on the discriminant 
 of the quadratic form  modulo $\ell$ for a positive density of primes.
  The above argument of Katz (from which the statement is deduced)
   asserts that, for $n$ odd and as $\ell$ varies ($\ell\nmid N_0$), the $\ell$-adic orthogonal group attached to $\mathcal H_n$ 
 forms the group of $\ell$-adic points of a single ``global'' quadratic form. This provides us with the control 
 we need on the discriminant of the quadratic forms modulo $\ell$. 
 If $\Delta\in\Qq$
  is the discriminant of the ``global'' quadratic form 
 then $\Delta$ modulo $\ell$ is a square for a density $1/2$ of the primes exactly 
 if $\Delta\in\Zz$ is not a square (otherwise the density is $1$, of course). This has to be done 
 for several quadratic forms simultaneously. The following section provides the precise property 
 we need.


\subsection{End of the proof} 

We first state a few preparatory results that will help us pick the set of primes of positive density 
needed to apply Theorem~\ref{product-sieve-statement}. The first lemma requires an application of the prime number theorem.

\begin{lemma}\label{lem:pnt}
For any fixed $A,B\geq 1$ and uniformly for $0<|a|\leq (\log x)^A$ we have that
$$ \sum_{ p \leq x} \left(  \frac a p\right) \geq -c_A \frac x{(\log x)^B}, $$
where $c_A$ is a positive constant which depends on $A$ only.
\end{lemma}
\begin{proof}
It is a well known fact that for any $b\neq 0$ with $b\not \equiv 3 \bmod 4$, the function $\big( \tfrac b {\cdot} \big)$ is a Dirichlet character. 
Taking $b=4a$, we note that
$$ \sum_{ p \leq x} \left(  \frac {4a} p\right) = \sum_{ p \leq x} \left(  \frac a p\right)+O(1),$$
since the only prime $p$ for which $\big(  \tfrac {4a} p\big)$ is not necessarily equal to
 $\big(  \tfrac {a} p\big)$ is $p=2$. Since $\big(  \tfrac {4a} {\cdot}\big)$ is a Dirichlet character, we obtain from Siegel's Theorem that
$$ \sum_{ p \leq x} \left(  \frac a p\right) = \eps_{a}\text{Li}(x)+ O_A\left( \frac x{(\log x)^A}\right), $$
where $\eps_a$ equals $1$ when the character is principal, and is zero otherwise. The result follows.
 
\end{proof}

\begin{lemma}\label{lem:pickLambda1}
Let $A$ be a ring satisfying $\Zz\subseteq A\subseteq \Qq$ such that only finitely 
many primes are invertible in $A$ (i.e.~$A$ is of the form $\Zz[1/N_0]$ for some integer $N_0$). 
Let $k\geq 1$ be an integer and let $(V_j,\Psi_j)_{1\leq j\leq k}$ be a sequence of 
(free of \emph{even} rank $r_j$)
  non-degenerate quadratic 
$A$-modules. For each $j$ let $\Delta_j$ be the discriminant of 
$(V_j,\Psi_j)$. For every odd prime $\ell\not\in A^\times$ coprime to $\prod_j\Delta_j$
 let $(V_{j,\ell},\Psi_{j,\ell})$ be the non-degenerate quadratic 
$\Fp_\ell$-vector space obtained by reduction modulo $\ell$. The lower density of primes 
$\ell$ for which the quadratic $\Fp_\ell$-vector spaces are \emph{simultaneously split}  
is at least $2^{-k}$.
\end{lemma}

\begin{proof}
Let $\Delta'_j=(-1)^{r_j/2}\Delta_j$. The question is that of the density of primes $\ell$
for which the $\Delta_j'$'s are simultaneously squares modulo $\ell$. In order to give a lower bound on this density, we note that
$$
 \# \{ \ell \leq x : \Delta'_j \equiv \square\, (\bmod \ell) \hspace{.2cm} \forall j \}
  \geq \sum_{\ell \leq x} \frac{\left( 1+ \left( \frac {\Delta_1'}{\ell}\right)\right)}2 \cdots \frac{\left( 1+ \left( \frac {\Delta_k'}{\ell}\right)\right)}2\,. 
  $$
(The difference between the left hand side and the right hand side comes from those $\ell$ dividing one of the $\Delta'_j$.) 
Expanding the right hand side gives that for $x$ large enough in terms of the $\Delta_j'$'s,
$$
 \frac{\pi(x)}{2^k}+  \frac 1{2^k}\sum_{g=1}^{k} \sum_{1\leq j_1<\dots < j_g \leq k} \sum_{\ell \leq x} \left( \frac{\Delta_{j_1}'\cdots \Delta_{j_g}'}{\ell}\right) \geq \frac{\pi(x)}{2^k} 
 +O_A \left(  \frac{x}{(\log x)^A} \right)\,,
 $$
by Lemma \ref{lem:pnt}. The lemma follows.
\end{proof}

\begin{lemma}\label{lem:pickLambda2}
Let $\Lambda_0$ be a set of primes of lower density $\delta_0$. Let $N_1$ and $N_2$ be 
positive natural numbers and let $p$ be a fixed prime number.
 The set of primes 
$$
\{\ell\in\Lambda_0\colon (p,\ell^{j}+1)=1\,,1\leq j\leq N_1,\,\, (p,\ell^{i}-1)=1\,,\, 1\leq i\leq N_2\}
$$
has lower natural density at least
$$
\delta_0-\frac{N_1(N_1+1)+N_2(N_2+1)}{2(p-1)}\,.
$$
\end{lemma}

\begin{proof}
 For any integer $i\geq 1$ let $\mu_i(\Fp_p)$ be the subgroup of $\Fp_p^\times$ 
 consisting of $i$-th roots of unity. Of course $\#\mu_i(\Fp_p)\leq i$ with equality if and only if 
 $i\mid p-1$. Let $\zeta\in\mu_i(\Fp_p)$ then the Prime Number Theorem in arithmetic 
 progressions asserts that the set of primes congruent to $\zeta$ modulo $p$ has density 
 $1/(p-1)$.  Thus the density of primes $\ell$ that are congruent to some element of $\mu_i(\Fp_p)$ 
 is $\#\mu_i(\Fp_p)/(p-1)$. Summing over $i$ we deduce that the upper density of primes 
 $\ell$ lying in $\cup_{1\leq i\leq N_2}\mu_i(\Fp_p)$ is at most $N_2(N_2+1)/(2(p-1))$. We handle 
 the condition $(p,\ell^{j}+1)=1$ (for $1\leq j\leq N_1$) in the same way,
  replacing roots of unity by roots of the 
 polynomial $X^{j}+1$ that are of cardinality at most $j$ in $\Fp_p$.
\end{proof}

We now have all the necessary ingredients to derive Theorem~\ref{th:main1} and 
Theorem~\ref{th:main2}. To begin with we invoke Theorem~\ref{KatzBigMonodromy} 
and Theorem~\ref{KBM2}. Thanks to Proposition~\ref{prop:monodromymodell} and to the
Goursat--Kolchin--Ribet Theorem (as stated e.g.~in~\cite[Prop.~5.1, Lemma 5.2]{Chav97})
 we deduce the existence 
of a set consisting of all prime numbers but finitely many of them such that condition 
(i) of Theorem~\ref{product-sieve-statement} is satisfied. Let us mention here 
that to deduce the existence 
of a Galois \'etale cover of the parameter variety with suitable properties 
from Proposition~\ref{prop:monodromymodell}, we invoke~\cite[Lemma $4.1$]{Jou09}. 

The set of primes obtained depends only on $k$ and 
on the dimension of the parameter variety 
(which in turn only depends on $d$ in the case of the
 quadratic twist family and on the degree $d$ of the 
 divisor $D$ in case of the pullback family). Using 
 Lemma~\ref{lem:pickLambda1} we can shrink this set of primes so that condition (ii) of 
 Theorem~\ref{product-sieve-statement} is satisfied. 
 This new set of primes $\Lambda_0(d,k)$ has lower density $\delta_0(d,k)$ at least 
 $2^{-k}$. We next apply Lemma~\ref{lem:pickLambda2} to $\Lambda_0(d,k)$. The integers $N_1$ 
 and $N_2$ (that are the dimensions of the alternating or symmetric 
 bilinear spaces involved) only depend on $d$ and $k$ so that for large enough $p$ the quantity 
 $$
 \delta_0(d,k)-\frac{N_1(N_1+1)+N_2(N_2+1)}{2(p-1)}
 $$
 is positive. Thus condition (ii) of Theorem~\ref{product-sieve-statement} is fulfilled for 
 the set of primes $\Lambda_{0}(d,k)$.
 The conclusion of 
Theorem~\ref{product-sieve-statement} follows and thus Corollary~\ref{cor4.7} applies to 
both the settings of Theorem~\ref{th:main1} and 
Theorem~\ref{th:main2} which finishes the proof of both these results.

 \par\medskip
 {\bf Acknowledgements.} We would like to thank E.~Kowalski for several useful discussions 
 and for pointing out to us relevant references where arguments needed in the proof of 
 Lemma~\ref{lem:operationscommute} are also used. The second author was supported by an NSERC Postdoctoral Fellowship.

\end{document}